\documentclass[12pt]{extarticle}

\usepackage{amsmath, amsthm, amssymb}
\usepackage[colorlinks=true,linkcolor=blue,urlcolor=blue]{hyperref}
\usepackage{graphicx}
\usepackage{caption}
\usepackage{mathtools}
\usepackage{enumerate}
 \usepackage{comment}

\usepackage{cleveref}
\usepackage{verbatim}
\usepackage{tikz,tikz-cd,tikz-3dplot}
\usetikzlibrary{matrix}
\usetikzlibrary{arrows}
\usepackage{algorithm}
\usepackage[noend]{algpseudocode}
\usepackage[normalem]{ulem}
\usepackage{subcaption}
\usepackage{multicol}
\oddsidemargin \evensidemargin
\usepackage{makecell}
\usepackage{array}
\usepackage{enumitem}

 \newtheorem{theorem}{Theorem}[section]

\newtheorem{proposition}[theorem]{Proposition}
\newtheorem{lemma}[theorem]{Lemma}
\newtheorem{corollary}[theorem]{Corollary}

\newtheorem{conjecture}[theorem]{Conjecture}

\newtheorem{cor}[theorem]{Corollary}
\theoremstyle{definition}

\newtheorem{remark}[theorem]{Remark}
\newtheorem{example}[theorem]{Example}

\newcommand{\PP}{\mathbb{P}}

\newcommand{\QQ}{\mathbb{Q}}
\newcommand{\CC}{\mathbb{C}}
\newcommand{\ZZ}{\mathbb{Z}}

\newcommand{\cH}{\mathcal{H}}
\newcommand{\cE}{\mathcal{E}}

\newcommand{\cF}{\mathcal{F}}
\newcommand{\cJ}{\mathcal{J}}

\title{\bf Eigenvector Varieties}

\author{Sandra Di Rocco, Bernd Sturmfels and Svala Sverrisd\'ottir}

\date{}
\begin{document}

\maketitle

\begin{abstract} Any linear space of square matrices has an associated eigenvector variety.
Its points are  eigenvectors of matrices from that linear space.
We present a systematic study of eigenvector varieties,
with focus on Lie algebras and Hamiltonians of quantum systems.
\end{abstract}





\section{Introduction}
Let $\cH \simeq \CC^d$ be a linear space of $n \times n$ matrices.
The general element of $\cH$ is an $n \times n$ matrix $H$ whose
entries are linear forms with complex coefficients in $d$ variables
$t=(t_1,\ldots,t_d)$. We assume that ${\rm det}(H)$ is non-zero: it is a non-zero  homogeneous polynomial of degree $n$. Each eigenvector of $H$ gives a point
$x=x(t)$ in projective space $\PP^{n-1}$. The multivalued algebraic function
$t\mapsto x(t)$ parametrizes a subvariety of $\PP^{n-1}$. This is the
\emph{eigenvector variety} $\cE(\cH)$.

\begin{example}[$n=4,d=5$] \label{ex:twisted}
Let $\cH$ be the five-dimensional space of $4 \times 4$ matrices
\[
H=\begin{bmatrix}
0 & t_1 & 0 & 0 \\
0 & 0 & t_1 & 0 \\
0 & 0 & 0 & t_1 \\
t_2 & t_3 & t_4 & t_5
\end{bmatrix}.
\]
Writing $\lambda=\lambda(t)$ for the eigenvalues of $H$, the four points in
$\PP^3$ representing eigenvectors are
\[
        x(t)= [\,t_1^3:t_1^2\lambda:t_1\lambda^2:\lambda^3\,]^T .
\]
This parametrization shows that the eigenvector variety is the {\em twisted cubic curve}
\[
\cE(\cH)\,=\,\{x\in\PP^3: x_1x_3-x_2^2=x_1x_4-x_2x_3=x_2x_4-x_3^2=0\}.
\]
Our derivation assumes $H$ is a general matrix of the given form. In
particular, $t_1\not=0$. Otherwise, every point $x\in\PP^3$ is a zero-eigenvalue
eigenvector for a suitable choice of $t_2,t_3,t_4,t_5$.
\end{example}

\begin{example}[$n=4,d=6$] \label{ex:skew}
Let $\cH$ be the space of skew-symmetric matrices of size $4\times4$. Its eigenvector
variety is a surface in complex projective $3$-space
that has no real points:
\begin{equation}\label{eq:norealpoints}
        \cE(\cH)\,=\,\{x\in\PP^3:x_1^2+x_2^2+x_3^2+x_4^2=0\}.
\end{equation}
Indeed, if $Hx=\lambda x$ then $x^THx=\lambda x^Tx=0$. For a general $H$, we have
$\lambda\ne0$,~so~$x^Tx=0$.
\end{example}

The \emph{full eigenvector variety} $\cF(\cH)$ is the
Zariski closure of all eigenvectors with nonzero eigenvalue of 
any matrix in $\cH$. 
We have $\mathcal{E}(\cH) \subseteq \mathcal{F}(\cH)$.
Equality holds
generically but the two can differ dramatically; see Example~\ref{ex:diagonal}. The
precise relationship is studied in Section~\ref{sec3}.

\begin{example}[$n=4,d=4$] \label{ex:diagonal}
Let $\cH$ be the space of diagonal $4\times4$ matrices. The entries of
the general $H$ are distinct and nonzero. Hence the eigenvector variety consists of four~points:
\[
\cE(\cH)=\{[1:0:0:0],[0:1:0:0],[0:0:1:0],[0:0:0:1]\}\subset\PP^3 .
\]
But, we have $\mathcal{F}(\cH)=\PP^3$, since
every vector is an eigenvector
of some special matrix in $\cH$.

\end{example}

To study the variety $\cE(\cH)$ in general, we parametrize the subspace
$\cH\subset\CC^{n\times n}$ by a matrix
\begin{equation}\label{eq:Hexpansion}
        H=t_1H_1+t_2H_2+\cdots+t_dH_d .
\end{equation}
Here each $H_i$ is a complex $n\times n$ matrix. The $n$ entries of the column
vector $H_i x$ are linear forms in $x=(x_1,\ldots,x_n)$. We write these $d$
vectors as the columns of the $n\times d$ matrix
\begin{equation}\label{eq:Hx}
        M(x)= [\,H_1x\,|\,H_2x\,|\,\cdots\,|\,H_dx\,] .
\end{equation}
Finally, we append the column vector $x$ on the right to obtain the
$n\times(d+1)$ matrix
\begin{equation}\label{eq:Hxx}
        \widehat M(x)= [\,H_1x\,|\,H_2x\,|\,\cdots\,|\,H_dx\,|\,x\,] .
\end{equation}
If $x\in\cE(\cH)$, then some vector of the form
$[t_1,t_2,\ldots,t_d,-\lambda]^T$ lies in the kernel of \eqref{eq:Hxx}. Thus all
$(d+1)\times(d+1)$ minors of \eqref{eq:Hxx} vanish on $\cE(\cH)$. This is only
a necessary condition in general. Example~\ref{ex:diagonal} shows that it is
not sufficient. It is sufficient for generic matrix spaces.

In Section \ref{sec2} we study the generic case. We provide an explicit description of $\cE(\cH)$:

\begin{theorem}\label{thm:generic}
Let $H_1,\ldots,H_d$ be generic $n\times n$ matrices. The radical ideal of its eigenvector variety is generated by the $(d+1)\times(d+1)$ minors of $\widehat M(x)$.
Its dimension is 
$\min({d,n})-1$ and its degree is $\binom nd$ for $d \le n$. For $d \ge 2$
it is irreducible. If $d = 1$ it~consists~of~$n$~reduced~points.
\end{theorem}

In the rest of the paper, $H_1,\ldots,H_d$ are arbitrary special
matrices, provided ${\rm det}(H)\not=0$. A key object is the incidence
variety $\mathcal{I}(\cH)$ in $\PP^d\times\PP^{n-1}$ defined by \eqref{eq:I(H)}. This is the closure of all points
$((t:\lambda),x)$ satisfying
$Hx=\lambda x$ and $\lambda\ne0$. Its image in $\PP^{n-1}$ is the full
eigenvector variety $\cF(\cH)$. The
incidence variety typically has many irreducible components. 
The vertical components are those that lie over special loci in $\PP(\cH)$. Only the horizontal components contribute
to $\cE(\cH)$. Their union is the horizontal incidence variety $\cJ(\cH)$. 
This is the focus of Section~\ref{sec3}. We show that the horizontal components are
in bijection with the irreducible factors of the characteristic polynomial $\chi_\cH(t,\lambda)$, which is
homogeneous of degree~$n$.

In Section \ref{sec4} we consider the $n \times d$ matrix $M(x)$ from \eqref{eq:Hx}. In particular we care about its rank strata. 
The dimension of an irreducible component $V$ in $\cE(\cH)$ is governed by the rank of $M(x)$ over $V$ along with the generic geometric multiplicity in $V$. 
In particular, when $\chi_\cH$ is squarefree, every irreducible component 
of $\cE(\cH)$ has dimension $\operatorname{rank} M(x) - 1$, 
where $x$ is a general point on that component. 
This provides us with tools to construct linear matrix spaces $\cH$ with low dimensional eigenvector varieties, using matrices $M(x)$ of low rank.

\begin{example}[$n=4,d=6$] \label{ex:skew2}
For the skew-symmetric matrices of Example~\ref{ex:skew} we have
$$  M(x) \,\,=\,\,
\begin{bmatrix}
 \phantom{-}x_2 & \phantom{-} x_3 & \phantom{-} x_4 & \,0 & \,0 & 0\, \\
-x_1 & \,0  &\, 0   & \phantom{-} x_3 & \phantom{-} x_4 & 0 \,\\
 \,0 & - x_1 &\, 0  & -x_2 & \, 0 & \!\phantom{-} x_4 \,\,\\
  \,0  & \,0 & -x_1 & \,0    & -x_2  & \!-x_3\,\,
\end{bmatrix}. \qquad
$$
This matrix has rank $3$. The rank does not drop over the quadric \eqref{eq:norealpoints}.
Thus the rank formula gives $\dim \cE(\cH)=2$.
The equation of the eigenvector surface 
$\cE(\cH)$ is the greatest common divisor
of all $4 \times 4$ minors of the augmented matrix $\widehat M(x)$ from (\ref{eq:Hxx}). See also
 Proposition \ref{prop:syzygy}.
\end{example}

The theory of Lie algebras is a fruitful source of special matrix spaces. In
Example~\ref{ex:skew} we had $\cH=\mathfrak{so}_4$, and the quadric
\eqref{eq:norealpoints} is the projective orbit of an isotropic line. In
Section~\ref{sec5} we express the eigenvector variety of 
a reductive Lie algebra $\mathfrak g$
in terms of its Lie group $G$.

\begin{theorem}\label{thm:intro-lie}
For a reductive matrix Lie algebra $\mathfrak g\subset\mathfrak{gl}(V)$ with
weight spaces $V_a$, we have
\begin{equation}\label{eq:LieEigVecs}
    \cE(\mathfrak g)\,\,=\,\,\bigcup_a \,\overline{G\cdot\PP(V_a)} .
\end{equation}
\end{theorem} 


In Section~\ref{sec6} we study the class of $\cJ(\cH)$ in the Chow ring of
$\PP^d\times\PP^{n-1}$. This computes dimension and degree
of $\cE(\cH)$
without 
knowing defining equations. For minuscule representations,
Proposition \ref{prop:minuscule-degree-formula}
gives a formula in terms of Chern classes of the closed orbit.

Our project arose from the study
of Hamiltonians in quantum chemistry, especially the one-body and two-body
operators in electronic structure theory \cite{FSS,Sve}. 
In Section~\ref{sec7} we examine such Hamiltonians and their
relation with Theorem \ref{thm:intro-lie}.
Classically, eigenvectors of fermionic one-body operators are Slater determinants~\cite[Exercise~2.15]{SO}. Hence the corresponding eigenvector variety is contained in the Grassmannian. We prove that 
equality~holds.

\begin{theorem}\label{thm:ferm1body}
Consider a fermionic system with $k$ electrons in $m$ orbitals. The eigenvector variety of the fermionic one-body operator
equals the Grassmannian, $\mathrm{Gr}(k,m)\subset\PP^{\binom{m}{k}-1}$.
\end{theorem}

For the fermionic two-body operators we prove that the eigenvector variety is irreducible. We give an upper bound on its dimension and conjecture that this bound is sharp, see Conjecture \ref{conj:dim}.
We also establish a bosonic analogue to Theorem \ref{thm:ferm1body}, see Theorem \ref{thm:sympower2}.



We are not aware of a previous systematic study of the eigenvector varieties $\mathcal{E}(\cH)$.
Ringel's notion in \cite{Rin} is different from ours: he 
considers common eigenvectors of all matrices in $\cH$. For generic $\cH$, his
varieties are defined by the $2\times2$ minors of \eqref{eq:Hx}, while ours are
defined by the $(d+1)\times(d+1)$ minors of \eqref{eq:Hxx}. Our eigenvector
varieties are dual in spirit to the Kalman varieties \cite{OS,Sam},
where one fixes a subspace of $\CC^n$ and asks for matrices
with an eigenvector in it.
We here fix a matrix space and ask which vectors are its eigenvectors.

\section{The Generic Case}
\label{sec2}

In this section we assume that 
$H_1,H_2,\ldots,H_d$ are generic $n \times n$ matrices with entries in $\CC$.
They span a  linear subspace $\mathcal{H}$ of dimension $d$ in the matrix space $\CC^{n \times n}$.
Equation (\ref{eq:Hexpansion}) is our parametric representation.
We regard $\mathcal{H}$ as a generic point in the Grassmannian ${\rm Gr}(d,n^2)$.

\begin{proof}[Proof of Theorem \ref{thm:generic}]
We work in the bihomogeneous coordinate ring $\CC[t,\lambda,x]$ of
$\PP^d\times\PP^{n-1}$, where $\PP^d=\PP(\mathcal H\oplus\CC)$.
The first factor records the matrix along with the eigenvalue, and the second factor records the
eigenvector.  Consider the saturated incidence ideal
\[
J\,=\,\langle Hx-\lambda x\rangle:(\langle x\rangle\cdot
\langle t,\lambda\rangle)^\infty .
\]

If \(d=1\), let \(\mathcal H=\langle H_1\rangle\), with
\(H_1\) generic.  Then \(H_1\) has \(n\) distinct eigenvalues, and
\(\mathcal E(\mathcal H)\) consists of the \(n\) eigenvectors
\([v_1],\ldots,[v_n]\).  The \(2\times 2\) minors of \([H_1x\,|\,x]\)
cut out precisely these points.  Hence \(\dim \mathcal E(\mathcal H)=0\)
and \(\deg \mathcal E(\mathcal H)=n\).
Next assume \(d\geq 2\), and abbreviate
\[
\mathcal U\,=\,\{\,((A:\lambda),x)\in
\PP(\CC^{n\times n}\oplus\CC)\times\PP^{n-1}
\mid (A-\lambda {\rm I}_n)x=0\,\}.
\]
For fixed \(x=[\hat x]\), the map
\((A,\lambda)\mapsto (A-\lambda {\rm I}_n)\hat x\) is surjective. As these maps
have constant rank \(n\) and vary algebraically with \(x\), their kernels form
a vector bundle of rank \(n^2-n+1\) over \(\PP^{n-1}\), whose projectivization is \(\mathcal U\). Hence \(\mathcal U\) is smooth and irreducible, of dimension $n^2 - 1$. Because the point \(p=[0_n:1]\) is not in the image of the projection
of \(\mathcal U\) to the first factor,  the hyperplanes through \(p\) pull back to a
base-point-free linear system on \(\mathcal U\).

Since \(\mathcal H\) is generic, \(\PP(\mathcal H\oplus\CC)\) is a general
projective \(d\)-plane through \(p\).  By Bertini,
\begin{equation}\label{eq:I(H)}
\mathcal I(\mathcal H)\,:=\,\mathcal U\cap
\bigl(\PP(\mathcal H\oplus\CC)\times\PP^{n-1}\bigr)
\end{equation}
is smooth and irreducible of dimension \((n^2-1)-(n^2-d)=d-1\), with saturated ideal \(J\).
For a general point \(t\in\PP(\mathcal H)\), the matrix \(H(t)\) has \(n\)
distinct nonzero eigenvalues, so the projection
\(\mathcal I(\mathcal H)\to\PP(\mathcal H)\) by forgetting
\(\lambda\) is dominant. Since \(\mathcal I(\mathcal H)\) is irreducible,
there are no vertical components.
Hence the ideal of \(\mathcal E(\mathcal H)\) is the elimination ideal
\(J\cap\CC[x]\). We claim 
\begin{equation}
\label{eq:idealidentity}
J\cap\CC[x]
\,=\,
I_{d+1}\bigl([\,H_1x\,|\,H_2x\,|\,\cdots\,|\,H_dx\,|\,x\,]\bigr).
\end{equation}
Indeed, writing \(y=(t_1,\ldots,t_d,-\lambda)^T\), we have
\(\widehat M(x)y=Hx-\lambda x\).  For any \((d+1)\times(d+1)\) submatrix \(B\) of \(\widehat M(x)\), the identity
\(\operatorname{adj}(B)B=\det(B)I\), applied to
\(y=(t_1,\ldots,t_d,-\lambda)^T\), shows that
\(\det(B)\langle t,\lambda\rangle\subseteq \langle Hx-\lambda x\rangle\).
Hence all maximal minors of \(\widehat M(x)\) lie in \(J\cap\CC[x]\).  
The contraction \(J\cap\CC[x]\) is prime. By elimination, it defines the
closed image of \(\mathcal I(\mathcal H)\to\PP^{n-1}\).  Its dimension is
\(\min(d,n)-1\). Indeed, the fiber over \(x=[\hat x]\) is
\(\PP(\ker\widehat M(\hat x))\), whose generic dimension is \(0\) for \(d<n\)
and \(d-n\) for \(d\ge n\).  On the other hand, by
standard determinantal theory for generic \(H_1,\ldots,H_d\), the ideal on
the right hand side of \eqref{eq:idealidentity} is prime and defines a
variety of the same dimension, with degree \(\binom nd\) for \(d\leq n\).
Since the former variety is contained in the latter, the two varieties, and
hence the two prime ideals, are equal.
\end{proof}

\begin{cor}
A generic subspace $\mathcal{H}$ has a proper eigenvector variety
if and only if $d < n$.
\end{cor}

For the rest of this section we assume  $d < n$.
The eigenvector variety $\mathcal{E}(\cH)$ has
dimension $d-1$ in $\PP^{n-1}$. It
consists of all points $x$ such that the 
matrix $\widehat M(x)$ in (\ref{eq:Hxx}) has linearly dependent columns.
After a linear change of coordinates, this matrix
is an arbitrary  $(d+1) \times n$ matrix whose
entries are generic linear forms in 
the unknowns $x=(x_1,\ldots,x_n)$.

 As observed above, for $d=1$ the variety $\mathcal{E}(\cH)$ is a set of $n$ points in $\PP^{n-1}$,
corresponding to the eigenvectors of the matrix $H_1$.
Next we analyze the first interesting cases for $d\geq 2$.

\begin{proposition}\label{prop:zwei}
The eigenvector curve $\mathcal E(\mathcal H)$ of a generic
pencil \(\PP(\mathcal H)\) has genus~$\binom{n-1}{2}$.
\end{proposition}

\begin{proof}
The ideal of $\mathcal{E}(\cH)$ is generated by the $3 \times 3$ minors
of an $n \times 3$ matrix of linear forms in $n$ variables. Since this ideal is resolved by the
Eagon-Northcott complex, its Hilbert series~is
$$ \frac{1 + (n-2) z + \binom{n-1}{2} z^2}{(1-z)^2} \,\, = \,\, 
1 \,+\,\sum_{t=1}^\infty\, \biggl[\,\binom{n}{2} \,t \,+\, 1-\binom{n\!-\!1}{2}\, \biggr] z^t .$$
The parenthesized Hilbert polynomial in $t$ reveals
the degree $\binom{n}{2}$ and the  genus $\binom{n-1}{2}$.
\end{proof}

\begin{example}
For $n=3$ the curve $\mathcal{E}(\cH)$ is a smooth plane cubic so it has genus~$1$.
For $n=4$ the curve $\mathcal{E}(\cH)$ is a space sextic of genus $3$, defined by
the $3 \times 3$ minors of a $ 3 \times 4$ matrix of generic
linear forms in $4$ variables. The canonical model for $\mathcal{E}(\cH)$
is a quartic curve in the plane $\PP^2$ with coordinates $(t_1:t_2:\lambda)$.
Its equation is  ${\rm det}(t_1 H_1 + t_2 H_2 - \lambda \,{\rm I}_4) = 0$.
\end{example} 


\begin{proposition}\label{prop:drei} The eigenvector surface $\mathcal{E}(\cH)$ 
of a generic $3$-dimensional space $\cH$ is a
surface of degree $\binom{n}{3}$ in $\PP^{n-1}$.
It is arithmetically Cohen-Macaulay, and its sectional genus~is
$$ g_{\rm sec}\bigl( \mathcal{E}(\cH) \bigr) \,\,=\,\, \frac{(n-3)(n-2) (2n+1)}{6}. $$
\end{proposition}

\begin{proof} The ideal is generated by the $4 \times 4$ minors
of an $n \times 4$ matrix of generic linear forms in $n$ variables.
Since this is resolved by the Eagon-Northcott complex, the Hilbert series equals
$$ \frac{1 + (n-3) z + \binom{n-2}{2} z^2 + \binom{n-1}{3} z^3}{(1-z)^3}. $$
The Hilbert series of a general hyperplane section is
\((1-z)\operatorname{Hilb}_{\mathcal E(\mathcal H)}(z)\), i.e. the same
numerator over \((1-z)^2\).
Proceeding as in the proof of  Proposition \ref{prop:zwei},
we compute the Hilbert polynomial of this curve. For a curve \(C\), the constant term of its Hilbert polynomial is
\(1-p_a(C)\), so the sectional genus equals \(1\) minus that
constant term: \(
g_{\rm sec}(\mathcal E(\cH))
=
\binom{n-2}{2}+2\binom{n-1}{3}\).
\end{proof}

The case
$n=d+1$ is especially interesting as the eigenvector variety $\mathcal{E}(\cH)$
is a linear determinantal hypersurface of degree $n$ in $\PP^{n-1}$, hence Calabi--Yau when smooth.
Conversely, a generic linear determinantal hypersurface arises as $\cE(\cH)$ for a suitable generic $\cH$.
The work of Reichstein and Vistoli \cite{RV} on
determinantal hypersurfaces implies the following result.

\begin{corollary}
For $n \geq 4$, the locus of generic eigenvector hypersurfaces $\mathcal{E}(\cH)$
has dimension $n^3-2n^2+1$ inside the projective space of all hypersurfaces
of degree $n$ in $\PP^{n-1}$.
\end{corollary}

In the case $d = 3$ and $n=4$, the above locus is a hypersurface in the space $\PP^{34}$
of all quartic surfaces.
The eigenvector surfaces $\mathcal{E}(\cH)$ are smooth
K3 quartic surfaces in $\PP^3$ of sectional genus $3$.
This Noether-Lefschetz divisor is defined
by an irreducible polynomial of degree
$320112$. This degree was found using intersection theory by
Leal, Lozano Huerta and Vite
 \cite[Theorem 2]{LLHV}. For $n \ge 4$, the degree of this locus can be numerically computed
using {\tt HomotopyContinuation.jl} \cite{BT}. 
Our code for this is available~at~\cite{code}. The repository contains a detailed explanation~of~the~algorithm.
Already for $n = 4$ this computation is time consuming. It terminated and verified the stated degree~$320112$.


\section{Incidence Varieties}
\label{sec3}

We return to the general setting from the Introduction. The input is a
$d$-dimensional linear space $\mathcal H$ of complex $n\times n$ matrices,
written as $H=\sum_{i=1}^d t_iH_i$, with $\det(H)$ not identically zero. The incidence variety $\mathcal{I} (\mathcal{H})$ is the
subvariety in $\PP^d \times \PP^{n-1} $
which is defined by the ideal
\begin{equation}
\label{eq:incidenceideal}
\qquad I_\mathcal{H} \,\,:=\,\,
\bigl\langle \,H x - \lambda x \,\bigr\rangle \,: \, 
\bigl( \langle \lambda \rangle \cdot \langle x \rangle \bigr)^\infty \quad \subset \,\, \,\CC[t,\lambda,x]. 
\end{equation}
The ideal on the left is generated by $n$ polynomials that are bilinear in the
$(d+1)+n$ variables $(t_1,\ldots,t_d,\lambda)$ and $x=(x_1,\ldots,x_n)$.
The saturation step removes all primary components whose
radical either contains $\lambda$ or contains all variables $x_1,\ldots,x_n$.
Geometrically, we remove any irreducible component that consists entirely of
singular matrices $H$ as well as the irrelevant component $\langle x\rangle$.
Hence the generic point of any irreducible component of $\mathcal{I}(\mathcal{H})$ is
an invertible matrix $H$ together with a fixed point $x$
of its linear map $\,\PP^{n-1} \rightarrow \PP^{n-1}, \,x\mapsto H x $. 

\begin{remark}
The image 
of the incidence variety $\mathcal{I}(\cH) \subset \PP^d \times \PP^{n-1}$
under its projection into the second factor 
is the full eigenvector variety $\mathcal{F}(\cH) \subset \PP^{n-1}$.
See also Proposition \ref{prop:eigenimage}.
\end{remark}

A distinguished element of the ideal $I_\mathcal{H}$ is the {\em characteristic polynomial} of our family:
$$ \chi_\mathcal{H}(t,\lambda) \,\, = \,\, {\rm det}( H - \lambda {\rm I}_n). $$
This is a homogeneous polynomial of degree $n$ in the $d+1$ variables $t_1,\ldots,t_d,\lambda$.

\begin{lemma}
The hypersurface defined by the characteristic polynomial $\chi_\mathcal{H}$ is the
closure of the image of the incidence variety $\mathcal{I}(\mathcal{H})$ under its
projection into the first factor $\PP^d$.
\end{lemma}

\begin{proof}
The $n$ coordinates of $(H - \lambda {\rm I}_n) x $ are in the ideal $I_\cH$.
Multiplying on the left with the adjoint of $H - \lambda {\rm I}_n$, we see that
${\rm det}(H-\lambda {\rm I}_n) \cdot x_i = \chi_\cH \cdot x_i $  is also in $I_\cH$ for each $i$.
Hence, by saturation, $\chi_\cH$ is in $I_\cH \cap \CC[t,\lambda]$.
Our hypothesis ${\rm det}(H) \not= 0$ ensures that 
$ \chi_\mathcal{H}$ is not divisible by $\lambda$.
Hence no irreducible factor gets removed by the saturation with $\lambda$.
We claim that the ideals $\langle \chi_\cH \rangle$
and $I_\cH \cap \CC[t,\lambda]$ have the same radical.
We already know that the former is contained in the latter.
The reverse inclusion follows from the Nullstellensatz, because
every point in the projection of  $\mathcal{I}(\mathcal{H})$ 
in $\PP^d$ must have at least one eigenvector $x \in \PP^{n-1}$.
\end{proof}

Consider a primary decomposition of the ideal $I_\cH$.
An associated prime is called {\em vertical} if it contains
a polynomial in $\CC[t]\backslash \{0\}$; otherwise it is called {\em horizontal}.
Geometrically speaking, an irreducible component of the incidence variety $\mathcal{I}(\cH)$
is horizontal if it maps dominantly onto $\PP^{d-1} = \PP(\cH)$.
We call the union of all horizontal components the {\em horizontal incidence variety},
and we denote it by $\mathcal{J}(\cH)$. We write $J_\cH$ for its ideal in $\CC[t,\lambda,x]$.

\begin{remark} \label{rmk:horizontal}
Our definition of horizontal is equivalent
to being horizontal with respect to the projection into the first factor
$\PP(\cH \oplus \CC)$. Indeed, if a component maps dominantly onto
$\PP(\cH)$, then its image in $\PP(\cH \oplus \CC)$ is contained in the
hypersurface $\{\chi_\cH=0\}$, and it has the same dimension.
Hence it is dense in an irreducible component of this hypersurface.
\end{remark}

One obvious method for computing $J_\cH$ is to examine a
primary decomposition of $I_\cH$ and to intersect each component with $\CC[t]$.
Then $J_\cH$ is the intersection of all components whose
intersection with $\CC[t]$ equals $\{0\}$.
A more efficient algorithm comes from {\em generic freeness}.
Namely, we can compute $J_\cH$ using the
Gr\"obner-based algorithm given by
Cid-Ruiz and Smirnov \cite{CRS}.
This rests on a Gr\"obner basis for
the ideal $\langle I_\cH \rangle$ in $\CC(t)[\lambda,x]$,
with respect to any term order on $\lambda,x$.
We clear denominators in $t$ to get polynomials in 
$\lambda,x$ whose coefficients are in $\CC[t]$.
Let $g$ be the least common multiple of all leading coefficients. Then,
\begin{equation}
\label{eq:getJ}
 J_\cH \,\, = \,\, I_\cH \,: \, \langle g \rangle^\infty . 
 \end{equation}
This technique is useful because it allows us to compute the eigenvector variety
by elimination. This method avoids a primary decomposition. 
Elimination via Gr\"obner bases becomes infeasible for larger systems. In such cases one can use numerical algebraic geometry. A numerical representation of $\mathcal{J}(\mathcal{H})$ can be found using~\texttt{nid}~in~\texttt{HomotopyContinuation.jl} \cite{BT}. We obtain a witness set for each irreducible component of $\mathcal{I}(\cH)$, and then numerically check whether each component is horizontal by rank computations on the Jacobian matrix.

\begin{proposition} \label{prop:eigenimage}
The eigenvector variety $\mathcal{E}(\cH)$ equals the image
of the horizontal incidence variety $\mathcal{J}(\cH) \subset \PP^d \times \PP^{n-1}$
under its projection into the second factor $\PP^{n-1}$.
\end{proposition}

\begin{proof}
We had defined $\mathcal{E}(\cH)$ through the eigenvectors
of the matrix $H = \sum_{i=1}^d t_i H_i$. These
eigenvectors have coordinates in the
algebraic closure of the rational function field $\CC(t)$.
Thus $\mathcal{E}(\cH)$ does not see eigenvectors
that are contributed by matrices given by special parameters $t \in \PP^{d-1}$. 
Only generic matrices in $\mathcal{H}$ matter.
Their eigenvectors appear in $\mathcal{E}(\cH)$.
In other words,
vertical components play no role in the definition of  the eigenvector variety.
On the other hand,  all eigenvectors arising from horizontal components do appear in $\mathcal{E}(\cH)$.
\end{proof}

In order to understand the components of $\mathcal{J}(\cH)$, we factor the characteristic polynomial:
\begin{equation}
\label{eq:charpolyfactor}
\chi_{\mathcal H}(t,\lambda)\,\,= \,\,P_1(t,\lambda)^{a_1} P_2(t,\lambda)^{a_2}\, \cdots \,P_k(t,\lambda)^{a_k}.
\end{equation}
The exponent $a_j$ is the algebraic multiplicity of the eigenvalues in the $j$th group
of eigenvalues. The factorization in (\ref{eq:charpolyfactor}) can be done
in $K[t,\lambda]$ for any subfield $K $ of $\CC$. In symbolic computations, one uses
$K = \QQ$. Numerical algebraic geometry deals with $K = \CC$.
The statements that follow will be valid for any field $K$. For ease of notation, we take $K = \CC$.

Let $r_j$ denote the rank of the matrix $H(t) - \lambda {\rm I}_n$ modulo $\langle P_j(t,\lambda) \rangle$.
Thus $n-r_j$ is the geometric multiplicity of any eigenvalue $\lambda = \lambda(t)$
in the $j$th group, assuming the matrix $H(t)$ is generic in $\cH$.
The familiar inequality between algebraic and geometric multiplicities is
\begin{equation}
\label{eq:alggeomineq}
 r_j\,\ge\, n-a_j. 
 \end{equation}

\begin{proposition}\label{prop:dominant-components}
The irreducible factors $P_j$ of $\chi_\mathcal{H}$ are in one-to-one correspondence with the irreducible components
$Y_j$ of 
$\mathcal{J}(\cH)$.
Moreover, the dimension of $\,Y_j$ equals $\,d+n-r_j-2$.
\end{proposition}

\begin{proof}
Let $\pi$ be the projection $((t :\lambda),x) \mapsto (t :\lambda)$.
On a dense open subset of the hypersurface $\{P_j=0\}$, the rank of
$H(t)-\lambda {\rm I}_n$ is constant and equal to $r_j$. Over such a point, the fiber
of $\pi$ is the linear space
$\mathbb P\bigl(\ker(H(t)-\lambda {\rm I}_n)\bigr) \simeq \PP^{n-r_j-1}$.
Hence, the inverse image of this open dense subset is a projective bundle.
Its closure is therefore a unique irreducible component $Y_j$ of $\mathcal{J}(\cH)$ on which $P_j$ vanishes. 
Now $\pi$ maps $Y_j$ dominantly to the hypersurface $\{P_j=0\}$ in
$\PP^d=\PP(\cH\oplus\CC)$. Hence the dimension of $Y_j$ equals $d-1$ plus
the dimension of the generic fiber of $\pi$. By the above description of the
fiber this~dimension~is~$n-1-r_j$.
\end{proof}

\begin{corollary}\label{cor:squarefree}
If the characteristic polynomial $\,\chi_\cH$ of the matrix family $\,\cH$ is squarefree
then every horizontal component of the
incidence variety $\,\mathcal{I}(\cH)$ has dimension $d-1$.
\end{corollary}

\begin{proof}
The algebraic multiplicity of each eigenvalue is $1$. The
geometric multiplicity must then also be $1$.  Equality holds in (\ref{eq:alggeomineq}), and we conclude that
$r_j = n-1$ for all $j$.
\end{proof}

The following example serves to illustrate the concepts and results seen in this section.

\begin{example}[$d=4,n=5$]\label{ex:addsym}
If $\cH$ is generic for these parameters, then the eigenvector variety $\mathcal{E}(\cH)$
is a quintic threefold in $\PP^4$. It is the image of the
incidence variety $\mathcal{I}(\cH) \subset \PP^4 \times \PP^4$ under projecting to $x$-space.
Note that $\mathcal{I}(\cH) = \mathcal{J}(\cH)$ is irreducible  of dimension~$3$.

More interesting scenarios emerge for special matrix families. Let us consider the family
{\small
\[
H \,\, = \,\,\begin{bmatrix}
 \,4 t_1 & t_4  &   0  &     0 &     0  \, \, \\
\,  4 t_3 & 3 t_1\!+\!t_2 & 2t_4 &    0  &    0 \,\,  \\
\,     0 &  3 t_3 &   2 t_1\!+\!2t_2 &  3 t_4 &   0 \, \, \\
\,      0 &  0 &     2 t_3 &  t_1\!+\!3t_2 & 4 t_4 \,\, \\
\,      0 & 0 &     0 &      t_3 &    4 t_2\,\, \end{bmatrix}.
\]
}

\noindent
The incidence variety $\mathcal{I}(\cH)$ has four irreducible
components. There is one vertical component, defined by
$\langle t_1-t_2,t_3,t_4, \lambda-4t_2 \rangle$.
Its removal gives us the horizontal incidence variety
$\mathcal{J}(\cH)$. This has three components, one for each factor
of the characteristic polynomial 
$$  \chi_\cH \,=\,
(\lambda-2t_1-2t_2) \bigl( \lambda^2-4(t_1+t_2) \lambda
+3t_1^2+10t_1t_2+3t_2^2-4t_3t_4 \bigr)
\bigl(\lambda^2-4(t_1+t_2) \lambda +16 (t_1t_2-t_3t_4) \bigr).
$$
Since $\chi_\cH$ is squarefree, all three
 horizontal components of $\mathcal{I}(\cH)$ are threefolds.
  Their images in $x$-space $\PP^4$ are irreducible varieties.
  We recognize these as
    multiple root loci of a quartic
$$ f(z) = x_1 z^4 + x_2 z^3 + x_3 z^2 + x_4 z + x_5 . $$
On the first variety, $f(z)$ has one root of multiplicity $3$.
On the second variety, $f(z)$ has two roots of order $2$.
On the third variety, $f(z)$ has one root of multiplicity $4$.
The third condition is the conjunction of the first two, so we
do not need it to describe $\mathcal{E}(\cH)$.
 Each of two components of $\mathcal{E}(\cH)$ is a surface, so the
dimension is one less than in the generic case.
We can state our conclusion succinctly as follows:
{\em The eigenvector variety $\mathcal{E}(\cH)$ is the
singular locus of the discriminant of $f(z)$}.
A general explanation of the matrix $H$ and why
discriminants make an appearance will be given 
in the setting of Lie algebras in Theorem~\ref{thm:sympower}.
\end{example}

Consider the projection from the horizontal incidence variety onto the eigenvector variety
\begin{equation}
\label{eq:pi}
\pi \,: \, \mathcal{J}(\cH) \, \rightarrow \, \mathcal{E}(\cH),\, \, \bigl((t:\lambda),x \bigr) \,\mapsto \, x. 
\end{equation}
The generic fibers $\pi^{-1}(x)$ are linear spaces in $\PP^d$. The dimensions of these linear spaces
dictate the dimensions of the various irreducible components of the eigenvector variety $\mathcal{E}(\cH)$.

\begin{proposition} \label{prop:fiber}
Suppose that $\chi_\cH$ is squarefree, 
consider any irreducible component $V$ of $\mathcal{E}(\cH)$,
and let $x$ be a generic point in $V$. Then we have
$\,{\rm dim}(V)\, =\, d-1-{\rm dim}(\pi^{-1}(x))$.
\end{proposition}

\begin{proof}
Since $V$ is smooth at $x$, the local dimension of $V$ at $x$ equals ${\rm dim}(V)$.
A general point in $\,\pi^{-1}(x)\,$ lies on a horizontal component. This has dimension $d-1$,
by Corollary \ref{cor:squarefree}. 
\end{proof}

\begin{corollary} \label{cor:d-2}
If \(\chi_{\mathcal H}\) is squarefree and 
the identity matrix is in $\cH$,
then
$\,{\rm dim} (\mathcal{E}(\cH)) \leq d-2$.
\end{corollary}

\begin{proof}
Fix $t_0$ in $\CC^d$ with ${\rm I}_n = H(t_0)$, and suppose that
 $(t:\lambda)$ is in a generic fiber  $\pi^{-1}(x)$. Then
 $(t + \mu t_0: \lambda + \mu)$ is in $\pi^{-1}(x)$ for all $\mu$.
This means that the fiber $\pi^{-1}(x)$ contains a line and is hence
positive-dimensional. The conclusion now follows from Proposition \ref{prop:fiber}.
\end{proof}

The linear space $\cH$ in Example \ref{ex:addsym} contains
the identity matrix ${\rm I}_5$, namely  for $t_0 = (1/4,1/4,0,0)$.
Corollary \ref{cor:d-2} explains why the two irreducible components
of $\mathcal{E}(\cH)$ have dimension  $d-2 = 2$, while the three components
of $\mathcal{J}(\cH)$ have dimension $d-1 = 3$.

\section{Low Rank Matrices}
\label{sec4}

The study of linear spaces of low rank matrices is a classical
subject at the interface of linear algebra and algebraic geometry \cite{EH}.
In this section we explore its relevance for eigenvector varieties.
We work in the space $\CC^{n \times d}$ of complex $n \times d$ matrices
and we consider any linear subspace ${\mathcal H}'$ of dimension $n$ in $\CC^{n \times d}$.
We write this subspace parametrically as in (\ref{eq:Hx}).
This justifies the notation $\mathcal{H}'$. Every  basis for $\mathcal{H}$
determines a basis for $\mathcal{H}'$. Both spaces are
given by the $d \times n \times n$ tensor with slices $H_1,H_2,\ldots,H_d$.
That tensor is our object of study.



The dimension of the eigenvector variety can be expressed in terms of linear algebra~data.

\begin{proposition}\label{prop:dimension} If $\chi_\cH $ is squarefree, then every irreducible component $V$ of
the eigenvector variety $\cE(\cH)$ satisfies $\dim V = \mathrm{rank}\, M(x)-1$,
where $x$ is a general point in $ V$.
\end{proposition}

\begin{proof}
By Corollary \ref{cor:squarefree} every horizontal component has dimension equal to $d-1.$
Since \(x\) is general, the open chart
\(\lambda\ne0\) is dense in the corresponding fiber of the
projection \eqref{eq:pi}.
Choose \(t_0\in\CC^d\) with \((t_0:1)\in\pi^{-1}(x)\).  Then
\((t_0+t:1)\in\pi^{-1}(x)\) holds precisely when \(M(x)t=0\).  Hence the affine chart
\(\lambda=1\) of the fiber is identified with the translated kernel \(t_0+\ker M(x)\). Thus the fiber $\pi^{-1}(x)$ has dimension $d-\mathrm{rank} \, M(x)$ and therefore $\dim V = (d-1)-(d-\mathrm{rank} \, M(x)) = \mathrm{rank} \, M(x)-1.$
\end{proof}

\begin{remark} Without the squarefree hypothesis, the same argument gives the formula
$$
\dim V =\mathrm{rank}\, M(x) + \mathrm{nullity}(H(t)-\lambda {\rm I}_n)-2,
$$ 
where $((t:\lambda),x)$ is a general point of the largest component of $\cJ(\cH)$ over $V$.\end{remark}

Proposition \ref{prop:dimension} can be used to construct linear spaces $\mathcal{H}$ 
with interesting eigenvector varieties $\mathcal{E}(\cH)$.
One starts with any linear space $\cH'$ in $\CC^{n \times d}$ of low rank matrices.
Every left syzygy of the general element $M(x)$ in $\cH'$ gives an equation that vanishes on 
  $\mathcal{E}(\cH)$,~as~follows:
 
\begin{proposition} \label{prop:syzygy}
On the constant-rank locus of $M(x)$, the full eigenvector variety $\mathcal{F}(\cH)$ is cut out by
polynomials $\sum_{i=1}^n x_i f_i(x)$ where
$(f_1,\ldots,f_n)$ is a syzygy in the left~kernel~of~$M(x)$. 
 \end{proposition}
 
 \begin{proof}
 A nonzero vector $\hat x$ lies in $\cF(\mathcal H)$ if and only if $\hat x\in \operatorname{colspan} M(\hat x)$. Equivalently if and only if every row vector in $\ker M(\hat x)^T$ annihilates $\hat x$. On the constant-rank locus, the left syzygy module specializes to this left kernel, so the equations are necessary~and~sufficient.
 \end{proof}

 We next discuss
two sources of low-rank spaces $\mathcal{H}'$.
For instance,  for odd integers $n$, we can take any
$n$-dimensional linear space of skew-symmetric $n \times n$ matrices. These determine $n$-dimensional linear spaces $\cH$ with eigenvector varieties of positive codimension.

\begin{example}[$n=5$]
Let $\mathcal{H}'$ be the following space of skew-symmetric 
$5 \times 5$ matrices:
\begin{equation}
\label{eq:Hx5}
 \bigl[ H_1 \,x\, |\, H_2\, x \,|\, \,\cdots \,\,|\, H_5 \,x \,\bigr]\,\,=\,\,
\begin{bmatrix} 
\,  0 & \phantom{-}x_1 &\phantom{-}x_3 &\phantom{-}x_4 &\phantom{-}x_5\, \, \\
\,     -x_1 & 0 & \phantom{-}x_2 &\phantom{-}x_5 &\phantom{-}x_1\, \, \\
\,     -x_3 & -x_2 & 0 & \phantom{-}x_3 &\phantom{-}x_2 \, \,\\
\,     -x_4 & -x_5 & -x_3 &  0 &\phantom{-}x_4 \, \,\\
\,     -x_5 & -x_1 & -x_2 & -x_4 & \,\,\, 0 \,\,
\end{bmatrix}. \qquad
\end{equation}
The incidence variety $\mathcal{I}(\cH)$ is irreducible of codimension $5$ in $\PP^5 \times \PP^4$.
We obtain its ideal by saturation from the five entries of the vector
$ \sum_{i=1}^5 t_i H_i x \, - \, \lambda x$.
By eliminating $t,\lambda$ from that ideal, we see that $\mathcal{F}(\cH)=\mathcal{E}(\cH)$ is a cubic
threefold in $\PP^4$.
Alternatively, we can compute it 
with the syzygy method in Proposition  \ref{prop:syzygy}.
Namely, the left kernel of (\ref{eq:Hx5}) is generated by
$$ f(x) \, = \,\bigl( x_1x_3+x_2x_4-x_2x_5 \, ,\,
      x_2 x_4 - x_3x_4 - x_3x_5 \, ,\,
      x_5^2   \, ,\,
     -x_1 x_2+x_1x_3-x_2x_5 \, ,\,
     x_1 x_3 + x_2 x_4 - x_3 x_5  \bigr).
$$
Using Proposition~\ref{prop:syzygy}, we obtain the cubic threefold in $\PP^4$ defined by
$$ \sum_{i=1}^5 x_i \cdot f_i(x) \,\, = \,x_1^2 x_3 + x_2^2 x_4 + x_1 x_3 x_4 
- x_2 x_3 x_4 - x_1 x_2 x_5 + x_1 x_3 x_5 - x_2 x_3 x_5. $$
\end{example}

Eisenbud and Harris \cite{EH} offer a systematic study of
low rank matrix spaces. We briefly recall
their theory of {\em compression spaces}.
Fix any vector space decompositions 
\begin{equation}
\label{eq:deco}
 \CC^d = U \oplus U' \quad {\rm and} \quad
\CC^n = V \oplus V' \quad {\rm with} \quad
{\rm dim}(U) = u \,\,\, {\rm and} \,\,\,{\rm dim}(V) = v. 
\end{equation}
Given this, we consider the space of all linear maps
$\CC^d \rightarrow \CC^n$ which map $U'$ into $V$.
This space has dimension $dv+nu-uv$, 
and each of its elements is
a linear map of rank at most $r = u+v$.
Suppose now that $ n \leq dv+nu-uv$. Then we can
choose an $n$-dimensional linear subspace $\mathcal{H}'$.
Any such compression space  $\mathcal{H}'$
in $\CC^{n \times d}$
determines a $d$-dimensional linear subspace $\mathcal{H}$
of $\CC^{n \times n}$
whose eigenvector variety $\mathcal{E}(\cH) $
has dimension at most $r-1$ in $\PP^{n-1}$.

In the following examples the
decompositions in (\ref{eq:deco}) are into coordinate subspaces.

\begin{example}[A threefold in $\PP^5$]
Let $n=d=6$ and $u=v=2$. Fix a
$6 \times 6$ matrix $M(x)$ whose lower right $4 \times 4$ block 
is zero, and whose other $20$ entries are
generic linear forms in $\CC[x]=\CC[x_1,x_2,x_3,x_4,x_5,x_6]$.
Then $\mathcal{H}'$ has rank $r = 4$. 
Its left kernel is generated by four quadratic
syzygies $f(x) = (f_1(x),\ldots,f_6(x))$.
The four cubics $\sum_{i=1}^6 x_i f_i(x)$ are the
$3 \times 3$ minors of a $3 \times 4$ matrix whose entries
are linear forms in $\CC[x]$.
These generate a prime ideal, whose variety is
an irreducible threefold of degree $6$.
It is not contained in the rank-drop locus and therefore
by Proposition \ref{prop:syzygy}, these cubics cut out the full eigenvector variety $\mathcal{F}(\cH)$. A computation verifies that it equals the eigenvector~variety~$\mathcal{E}(\cH)$.
\end{example}

\begin{example}[Eigenvector curves]
Let $u=v=1$ and $n,d \geq 2$. 
Fix an  $n$-dimensional space $\mathcal{H}'$ of
$n \times d$ matrices with zero entries outside the first row or first column.
This specifies a $d$-dimensional subspace $\mathcal{H}$
of $\CC^{n \times n}$ whose eigenvector variety in $\PP^{n-1}$ is finite or a curve.
\end{example}

\section{Lie Algebras}
\label{sec5}

Lie algebras are a natural source of linear spaces of square matrices.
We fix the action of 
a \(d\)-dimensional
connected reductive Lie group  \(G\) on $\CC^n$.
The induced Lie algebra representation is a map
$\rho:\mathfrak g\to \mathfrak{gl}_n$.
The image, \(\rho(\mathfrak g)\subseteq \mathfrak{gl}_n\), is a linear
space of \(n\times n\) matrices. 
Since the eigenvector variety depends only on the chosen representation
through its image, we will often suppress $\rho$ and write $\mathfrak g$
for the matrix Lie algebra $\rho(\mathfrak g)\subseteq \mathfrak{gl}_n$. 
Thus, in our dimension counts, $\dim \mathfrak g$ means $\dim \rho(\mathfrak g)$. These dimensions agree when the~representation~is~faithful.

\begin{proposition}\label{prop:Ginv}
The eigenvector variety \(\mathcal E(\mathfrak g)\subseteq \mathbb P^{n-1}\) is invariant under the $G$-action.
\end{proposition}

\begin{proof}
For $H\in\mathfrak g$, the conjugate,
$
gHg^{-1}
$, with an element $g \in G$
is also in $\mathfrak g$. Additionally, if $v$ is an eigenvector of $H$, then
$g\cdot v$ is an eigenvector of $gHg^{-1}$. Hence
$\mathcal E(\mathfrak g)$ is $G$-invariant.
\end{proof}

We follow \cite[Section~10.2]{MS} in discussing the weight decomposition for
representations of reductive Lie groups. Fix a maximal torus $T$ in the group $G$,
and let $r$ be the rank of $T$.
By restricting the representation $V=\CC^n$ to $T$, we obtain a
decomposition into weight spaces:
\begin{equation}
\label{eq:weightspaces}
 V\,\,=\,\,\bigoplus_{a\in\mathbb Z^r} V_a. 
 \end{equation}
The subspace $V_a$ consists of those vectors in $V$ that are scaled by $t^a$,
that is $t\cdot v=t^a v$ for all $t\in T$. These
spaces are called {\em weight spaces}, and $a$ is the corresponding weight. The multiplicity of the weight $a$ is defined to be the dimension of $V_a$.
For a general matrix $H$ in the Lie algebra $\mathfrak{t}$ of 
the torus $T$, distinct weights give distinct
eigenvalues for $H$. Hence the eigenspaces of $H$ are
exactly the weight spaces $V_a$.

\begin{proof}[Proof of Theorem \ref{thm:intro-lie}]

Recall that, for a complex reductive group $G$, the regular semisimple locus is dense in \(\mathfrak g\), and every regular
semisimple element is \(G\)-conjugate to an element of
the Lie algebra \(\mathfrak t\) of the maximal torus $T$. Thus a
general element of \(\mathfrak g\) is conjugate to an element of \(\mathfrak t\). By
choosing this element outside the finitely many hyperplanes where two distinct
weights coincide, the eigenspaces are exactly the weight spaces \(V_a\). Hence
 the eigenvectors of a general element of $\mathfrak g$ lie on the $G$-orbits of the projectivized weight spaces.
Conversely, every nonzero vector in a weight space $V_a$ is an eigenvector
of a general element of $\mathfrak t$. By Proposition~\ref{prop:Ginv}, every
point of $G\cdot \PP(V_a)$ is therefore an eigenvector of an element of
$\mathfrak g$. Taking closures and using the finiteness of the set of weights gives the equality.
\end{proof}


\begin{example}[Skew-symmetric matrices]\label{ex:skew3}
 In Example \ref{ex:skew} and \ref{ex:skew2} we considered $4\times 4$ skew-symmetric matrices, which is the $6$-dimensional Lie
algebra $\mathfrak{so}_4$ of the special orthogonal group $\mathrm{SO}_4$.
We decompose $\CC^4=E\oplus F$ into complementary maximal isotropic
subspaces, with bases $e_1,e_2$ of $E$ and $f_1,f_2$ of $F$. With respect
to this basis of $\CC^4$, a maximal torus is
$$
T\,\,=\,\,\bigl\{\operatorname{diag}(t_1,t_2,t_1^{-1},t_2^{-1}) : t_1,t_2\in\CC^*\bigr\}
\,\, \subseteq\,\, \mathrm{SO}_4.
$$
The standard representation decomposes into four one-dimensional weight spaces
$$
E \oplus F \,\,=\,\, \CC e_1\,\oplus\, \CC e_2\,\oplus \,\CC f_1\,\oplus\, \CC f_2.
$$
Here $r=2$ and the weights are $(\pm 1,0),(0,\pm 1) \in \ZZ^2$.
Since the weight vectors lie in $E$ or $F$, they are isotropic.
Since  $\mathrm{SO}_4$ preserves isotropy and acts transitively on the isotropic quadric
(\ref{eq:norealpoints}),
 the eigenvector variety is a single $\mathrm{SO}_4$-orbit:
$ \mathcal E(\mathfrak{so}_4) = \overline{\mathrm{SO}_4\cdot [e_1]} = \mathrm{SO}_4\cdot [e_1] $.
\end{example}

The given representation $\rho:\mathfrak g\to \mathfrak{gl}_n$ is called {\em minuscule} if it is irreducible and the Weyl group of $G$
acts transitively on the weights $a$ that occur in (\ref{eq:weightspaces}).
This property implies that each $V_a$ is one-dimensional.
For an introduction to minuscule  representations see \cite{Ses}.
The standard representation $\CC^{2m}$ of $\mathfrak{so}_{2m}$ is minuscule.
The case $m=2$ appears in Example \ref{ex:skew3}.

\begin{proposition} \label{thm:minuscule}
 The eigenvector variety  of a minuscule  representation is the projectivized $G$-orbit 
 of any highest weight vector in $\CC^n$. In fact, this $G$-orbit is Zariski closed in $\PP^{n-1}$.
\end{proposition}

\begin{proof}
We use the formula  given in (\ref{eq:LieEigVecs}).
The minuscule  hypothesis ensures that the orbits $G \cdot \PP(V_a)$
are the same for every $a$. Hence  $\mathcal{E}(\mathfrak{g}) = \overline{G \cdot \PP(V_a)}$,
where $a$ is a highest weight. Then $V_a = \CC v$ for any highest weight vector $v$.
Writing $[v] \in \mathbb{P}^{n-1}$ for the image of $v$, we conclude 
that $\mathcal E(\mathfrak{g}) = G \cdot [v] = \overline{G \cdot [v]}$. The 
last equality holds for every irreducible representation, since the 
orbit of a highest weight vector is closed in $\mathbb{P}^{n-1}$; 
see e.g.~\cite[Section~1.4]{GIT}.
\end{proof}

One benefit of the representation-theoretic viewpoint is that it immediately
produces many examples. Starting from any representation $\CC^m$ of $G$,
we may form new representations of $G$ by applying standard linear algebra operations,
such as tensor powers, symmetric powers, and exterior powers. For instance,
 consider the standard representation $\CC^m$ of $\mathfrak{gl}_m$. The induced Lie algebra action on the exterior power space $\wedge^k \CC^m \cong \CC^{\binom{m}{k}}$ is defined as
\begin{equation}
\label{eq:exterioraction}
g \cdot (v_1 \wedge \cdots \wedge v_k) \,\,= \,\,\sum_{i = 1}^k v_1 \wedge \cdots
\wedge v_{i-1} \wedge g v_i \wedge v_{i+1} \wedge \cdots \wedge v_k. 
\end{equation}
The map $\rho: \mathfrak{gl}_m \to \mathfrak{gl}(\wedge^k \CC^m)$ is obtained from
(\ref{eq:exterioraction}) by linearity.
We write $\mathcal A_{k,m}$ for its image. 

For \(1\leq k\leq m-1\), the exterior power representation of
\(\mathfrak{gl}_m\) on \(\wedge^k\mathbb C^m\) is faithful. Hence $\mathcal{A}_{k,m}$ is
 an $m^2$-dimensional linear space of $\binom{m}{k} \times \binom{m}{k}$ matrices.
These are called \textit{additive compound matrices} in the linear algebra literature.
Suppose $g$ is a diagonalizable $m \times m$ matrix with eigenvectors $v_1,v_2,\ldots,v_m$
and eigenvalues $\lambda_1,\lambda_2,\ldots,\lambda_m$. Its additive compound matrix $\rho(g)$ has the
$\binom{m}{k}$ eigenvectors $v_{i_1} \wedge \cdots \wedge v_{i_k}$ ~with~eigenvalues~$\lambda_{i_1}+\cdots+\lambda_{i_k}$. For \(k=m\), the representation is the trace character and the image is
one-dimensional.

\begin{corollary}\label{cor:addcomp}
    The eigenvector variety for the space $\mathcal{A}_{k,m}$
    of additive compound matrices is the Grassmannian in its Plücker embedding. In symbols,
     $ \mathcal{E}(\mathcal{A}_{k,m}) = \mathrm{Gr}(k,m) \subseteq \PP^{\binom{m}{k} - 1}$.
\end{corollary}

\begin{proof}
The $\mathfrak{gl}_m$-representation $\wedge^k \CC^m$ is minuscule,
with highest weight vector $v=e_1 \wedge \cdots \wedge e_k$.
The orbit of $[v]$ in
$ \PP^{\binom{m}{k} - 1}$ is the Grassmannian  $\mathrm{Gr}(k,m)$.
Now apply Proposition~\ref{thm:minuscule}.
\end{proof}

\begin{remark} \label{rmk:othergrassmannians}
The same argument works for the other classical Lie algebras. For example,
\[
\mathrm{SO}_m\cdot [e_1\wedge\cdots\wedge e_k] \,=\, \mathrm{OGr}(k,m),
\qquad
\mathrm{Sp}_{2m}\cdot [e_1\wedge\cdots\wedge e_k] \,=\, \mathrm{IGr}(k,2m).
\]
Here \(\mathrm{OGr}(k,m)\) and \(\mathrm{IGr}(k,2m)\) are the orthogonal and
symplectic Grassmannians of isotropic \(k\)-planes.
Note that
\(\mathrm{IGr}(m,2m)\) is the Lagrangian Grassmannian. Thus the eigenvector
varieties for the corresponding additive compound matrices are 
also Grassmannians.
\end{remark}

We next consider the symmetric power representation $\rho:\mathfrak{gl}_m\to\mathfrak{gl}(\mathrm{Sym}^k \CC^m)$.
This representation is not minuscule. Its eigenvector variety is more interesting than those in Proposition \ref{thm:minuscule}.
 We denote the image of $\rho$ by
 $\mathcal{S}_{k,m}$.  This is an $m^2$-dimensional linear space of $\binom{m + k - 1}{k} \times \binom{m + k - 1}{k}$ matrices.
The $5 \times 5$ matrix $H$ from Example \ref{ex:addsym} parametrizes this space for $m = 2$ and $k = 4$.
We next record an explicit formula for the matrices in $\mathcal{S}_{k,m}$.

The element in $\mathfrak{gl}_m$ given by
an $m \times m$ matrix $t = (t_{ij})$ is the differential operator
\begin{equation}
\label{eq:operator}
 \sum_{i=1}^m \sum_{j=1}^m\, t_{ij} \,z_j \frac{\partial}{\partial z_i} .
 \end{equation}
Let $Z$ be the lexicographically ordered list of all  $\binom{m + k - 1}{k}$ monomials 
$z^a$ of degree $k$ in the variables $z_1,z_2,\ldots,z_m$.
The matrix $H = H(t)$ that parametrizes  the space $\mathcal{S}_{k,m}$ has
 its rows and columns labeled by $Z$.
To compute the column of $H$ indexed by $z^a$ we apply
the operator (\ref{eq:operator}) to $z^a$. The coefficient of
$z^b$ in the resulting polynomial is the entry of $H$ in row $z^b$.

\begin{example}[$k\!=\!m\!=\!3$] \label{ex:33}
Here $d= m^2 = 9$ and $n = \binom{m + k - 1}{k}= 10$. The monomial vector~is
$$ Z = \bigl(z_1^3,  z_1^2 z_2, z_1^2 z_3, z_1 z_2^2, z_1 z_2 z_3, z_1 z_3^2, z_2^3, z_2^2 z_3, z_2 z_3^2, z_3^3 \bigr). $$
The linear space $\mathcal{S}_{3,3}$ consists of the following matrices with rows and columns labeled by $Z$:
{\footnotesize
\[
H \,= \, \begin{bmatrix}
3 t_{11} & t_{21} & t_{31} & 0 & 0 & 0 & 0 & 0 & 0 & 0 \\
3 t_{12} & 2 t_{11} \!+\! t_{22} & t_{32} & 2 t_{21} & t_{31} & 0 & 0 & 0 & 0 & 0 \\
3 t_{13} & t_{23} & 2 t_{11} \!+\! t_{33} & 0 & t_{21} & 2 t_{31} & 0 & 0 & 0 & 0 \\
0 & 2 t_{12} & 0 & t_{11} \!+\! 2 t_{22} & t_{32} & 0 & 3 t_{21} & t_{31} & 0 & 0 \\
0 & 2 t_{13} & 2 t_{12} & 2 t_{23} & t_{11} \!+\! t_{22} \!+\! t_{33} & 2 t_{32} & 0 & 2 t_{21} & 2 t_{31} &  0 \\
0 & 0 & 2 t_{13} & 0 & t_{23} & t_{11} \!+\! 2 t_{33} & 0 & 0 & t_{21} & 3 t_{31} \\
0 & 0 & 0 & t_{12} & 0 & 0 & 3 t_{22} & t_{32} & 0 & 0 \\
0 & 0 & 0 & t_{13} & t_{12} & 0 & 3 t_{23} & 2 t_{22} \!+\! t_{33} & 2 t_{32} & 0 \\
0 & 0 & 0 & 0 & t_{13} & t_{12} & 0 & 2 t_{23} & t_{22} \!+\! 2 t_{33} & 3 t_{32} \\
0 & 0 & 0 & 0 & 0 & t_{13} & 0 & 0 & t_{23} & 3 t_{33} 
\end{bmatrix}.
\]
}
\end{example}

The eigenvectors of $H$ are constructed from the eigenvectors of $t$ as follows.
Suppose that $t$ is diagonalizable with eigenvectors $v_1,v_2,\ldots,v_m$
and eigenvalues $\lambda_1,\lambda_2, \ldots,\lambda_m$.
Each eigenvector $v_i$ is identified with a linear form
$v_i \cdot z = v_{i1} z_1 + v_{i2} z_2 + \cdots + v_{im} z_m$.
For any monomial $z^a = z_1^{a_1} \cdots z_m^{a_m}$ in
the list $Z$ we consider the product of corresponding linear forms
\begin{equation}
\label{eq:vz}
 \prod_{i=1}^m (v_i \cdot z)^{a_i} \,\, \in \,\, \mathrm{Sym}^k \CC^m \,\simeq \, \CC^n.
\end{equation} 
This element in our vector space  is an eigenvector of the matrix $H$ with eigenvalue $\sum_{i=1}^m a_i \lambda_i$.
The coefficients of (\ref{eq:vz}) are algebraic functions
of the matrix entries $t_{ij}$. These functions are rational when $a_1 = a_2 = \cdots = a_m $.
Here is an explicit example to demonstrate this point.

\begin{example}[$k\!=\!m\!=\!3$] \label{ex:33eigenvectors}
For $a=(1,1,1)$, the construction in \eqref{eq:vz} gives a ternary cubic whose
coefficients are cubic polynomials over
$\QQ$ in the entries $t_{ij}$. This cubic is an
eigenvector of the $10\times10$ matrix $H$ in Example~\ref{ex:33}, with eigenvalue
$\operatorname{trace}(t)=\lambda_1+\lambda_2+\lambda_3$. The 
eigenvector equation is checked by
applying the differential operator \eqref{eq:operator} to the cubic. 
\end{example}

The eigenvector variety for $\mathcal S_{k,m}$ is a union of \textit{refined Chow varieties}.
These are indexed by integer partitions $\mu \vdash k$, and they comprise the 
associated products of linear forms:
$$ 
\operatorname{Ch}_\mu(\PP^{m - 1}) \,\,= \,\,\overline{\{
\,[\ell_1^{\mu_1} \cdots \,\ell_s^{\mu_s}] \, :\, [\ell_i] \in \PP^{m- 1}\,\}}.
$$
These varieties have inclusion patterns compatible with the refinement order of partitions. That is $\operatorname{Ch}_\mu(\PP^{m - 1}) \subseteq \operatorname{Ch}_\sigma(\PP^{m - 1})$ if and only if $\mu \le \sigma$. The Chow variety $\operatorname{Ch}_{1^k}(\PP^{m - 1})$ therefore contains all the other refinements, so we distinguish it and call it the \textit{full Chow variety}. 
For a study of this variety see \cite[Section 4.2.H]{GKZ}.
It is cut out by {\em Brill's equations}.
At the other end of the spectrum is the {\em Veronese variety}, which is
 $\operatorname{Ch}_{k}(\PP^{m - 1}) = \nu_k(\PP^{m- 1})$.
 

\begin{theorem} \label{thm:sympower}
The eigenvector variety for the symmetric power representation of  $\mathfrak{gl}_m$ equals
\begin{equation}
\label{eq:sympower}
\mathcal{E}(\mathcal{S}_{k,m})
\,\,=\,
\bigcup_{\substack{\mu \vdash k\\ \ell(\mu) \leq m}}
\operatorname{Ch}_{\mu}(\PP^{m - 1}).
\end{equation}
The irreducible components are given by
 the {\em maximal} partitions \(\mu\vdash k\) of length at most $m$. 
 Here $\mu$ being maximal means that  the length  $\ell(\mu)$
 is the minimum of $m$ and $k$.
  \end{theorem}

\begin{proof}
Each weight space in (\ref{eq:weightspaces})
 is one-dimensional and spanned by a monomial:
 \[
\mathrm{Sym}^k \CC^m
\,=\,
\bigoplus_{|a|=k} \CC z^a.
\]
By Theorem \ref{thm:intro-lie},  the eigenvector variety 
is the union of the orbit closures
$\overline{\mathrm{GL}_m\cdot [z^a]}$.
This depends  only on the partition 
$\mu=(\mu_1,\ldots,\mu_s)$
obtained by
sorting the nonzero entries of $a$:
\[
\overline{\mathrm{GL}_m\cdot [z^a]}
\,\,=\,\,
\operatorname{Ch}_{\mu}(\PP^{m - 1}).
\]
Indeed, the orbit consists of all products $[\ell_1^{\mu_1} \cdots 
\ell_s^{\mu_s}]$ with $\ell_1, \ldots, \ell_s$ linearly 
independent. The closure allows the linear forms to coincide. 
The identification $\overline{{\rm GL}_m \cdot [z^a]} = 
\mathrm{Ch}_\mu(\mathbb{P}^{m-1})$ follows from the fact that 
$\mathrm{Ch}_\mu(\mathbb{P}^{m-1})$ is irreducible and is the 
unique ${\rm GL}_m$-orbit closure containing $[z^a]$; 
see~\cite[Section~4.2.H]{GKZ} for the case $\mu = 1^k$ and 
\cite[Proposition~1.4]{LP} for the general case. Since a monomial in \(m\) variables has at
most \(m\) nonzero exponents, the partitions
$\mu$ that occur are precisely those with \(\ell(\mu)\le m\).

The refinement order on partitions gives the inclusions among refined Chow
varieties. Therefore only the maximal partitions are needed.
Since their
refined Chow varieties are irreducible and not contained in one another, they
give the irreducible decomposition.
\end{proof}

We conclude with a discussion of the case $k \le m$.
There is a unique maximal partition, namely  $\mu = 1^k$, and
 the eigenvector variety is the full Chow variety $\mathcal{E}(\mathcal{S}_{k,m})= \mathrm{Ch}_{1^k}(\PP^{m - 1})$,
 whose points are products of linear forms.
  This variety is irreducible, and it is the closure of a single $\mathrm{GL}_m$-orbit. 
The full Chow variety $ \mathrm{Ch}_{1^k}(\PP^{m - 1})$ is cut out by the 
classical equations that were found by Brill   \cite[Section 4.2.H]{GKZ}.

 \begin{example}[$m=3$]
 If $k=2$ then the eigenvector variety is a hypersurface in $\PP^5$, namely
 the set of all symmetric $3 \times 3$ matrices of rank $\leq 2$. Consider $k=3$, so we are interested in the eigenvectors of 
 the $10 \times 10$ matrix $H$ in Example \ref{ex:33}.
 The variety   $\mathcal{E}(\mathcal{S}_{3,3})= \mathrm{Ch}_{111}(\PP^{2}) \subset \PP^9$
 has dimension $6$ and degree $15$. Example  \ref{ex:33eigenvectors}
 gives a parametric representation.
  The prime ideal of $\mathcal{E}(\mathcal{S}_{3,3})$ is generated by
 $35$ quartics in $10$ variables, namely the Brill equations.
  \end{example}
 
 
\section{Multidegree of the Horizontal Incidence Variety}
\label{sec6}

This section computes dimensions and degrees of eigenvector varieties from
the Chow class of the horizontal incidence variety.  This
is useful because it avoids computing defining equations for the eigenvector
variety itself.  The same class also controls what happens after 
replacing the matrix space
\(\mathcal H\) by a sufficiently general linear subspace of $\cH$.
We fix the projections
\[   \pi_1:\PP^d\times\PP^{n-1}\longrightarrow \PP^d \qquad {\rm and}\qquad \pi_2:\PP^d\times\PP^{n-1}\longrightarrow \PP^{n-1}.
\]

 Throughout this section we assume that the horizontal incidence 
variety $\mathcal J(\mathcal{H})$ is equidimensional of dimension $d-1$. 
Let $Y_1, \ldots, Y_r$ denote the irreducible components of 
$\mathcal J(\mathcal{H})$. Thus each $Y_j$ has codimension $n$ in 
$\mathbb{P}^d \times \mathbb{P}^{n-1}$. This holds in particular  when 
$\chi_\mathcal{H}$ is squarefree, by Corollary~\ref{cor:squarefree}. In practice, 
one can verify this by checking that 
$\mathrm{Res}_\lambda(\chi_\mathcal{H}, 
\partial_\lambda \chi_\mathcal{H}) \neq 0$ in $\mathbb{C}[t]$, 
which is an open condition on $\mathcal{H}$. When $\chi_\mathcal{H}$
has repeated factors, the formulas below remain valid if 
$[\mathcal J(\mathcal{H})]$ is interpreted as a cycle, with components 
counted with multiplicity as in~\eqref{eq:charpolyfactor}.  We also assume that the components whose
images under \(\pi_2\) have maximal dimension have distinct images.  Without
this assumption, the same formulas compute the degree of the corresponding
cycle, with repeated images counted with multiplicity.

The Chow ring of the ambient product of two projective spaces is
\begin{equation}
\label{eq:chowring}
A^*(\PP^d\times \PP^{n-1}) \,\,
 =\,\,
\ZZ[u,v]/\langle u^{d+1},v^n\rangle ,
\end{equation}
where \(u\) and \(v\) are the hyperplane classes.
Since \(Y_j\) has codimension \(n\), we write
\[
 [Y_j] \,\,
 =
 \sum_{a=1}^{\min\{d,n\}}
 \gamma_a(Y_j)u^av^{n-a}.
\]
The term with \(a=0\) vanishes because \(v^n=0\).  
By the additivity of Chow
classes,
\begin{equation}
\label{eq:classesadditive}
[\mathcal J(\mathcal H)]
\,\, =\,\,
 \sum_{j=1}^r [Y_j]
\,\, =
 \sum_{a=1}^{\min\{d,n\}} \!\!
 \gamma_a(\mathcal J(\mathcal H))u^av^{n-a},
\qquad
\gamma_a(\mathcal J(\mathcal H))\,=\,
\sum_{j=1}^r\gamma_a(Y_j).
\end{equation}

Using intersection theory notation, the
 coefficient \(\gamma_a(Y_j)\) can be written as follows:
\begin{equation}
\label{eq:gamma-intersection}
\gamma_a(Y_j)
\,=\,
\int_{\PP^d\times\PP^{n-1}}[Y_j]\, u^{d-a}v^{a-1}.
\end{equation}
This has the following geometric meaning:
The number $\gamma_a(Y_j)$
counts, with multiplicity, the intersection of \(Y_j\) with
\(d-a\) general hyperplanes from the first factor and \(a-1\) general
hyperplanes from the second factor.  These coefficients are known as multidegrees.
They can be computed symbolically in \texttt{Macaulay2} \cite{M2} with the
command \texttt{multidegree}, or numerically by solving the intersection
problems in \eqref{eq:gamma-intersection}, for instance with
\texttt{HomotopyContinuation.jl}~\cite{BT}.

\begin{example}[Generic case]
Let \(d=n\), and let \(\mathcal H\in {\rm Gr}(n,n^2)\) be generic 
in the sense of
Section~\ref{sec2}.  Then \(\mathcal J(\mathcal H)\) is the complete
intersection defined by the \(n\) bilinear equations
\((H(t)-\lambda {\rm I}_n)x=0\) in \(\PP^n\times\PP^{n-1}\).  Each equation has
class \(u+v\), and hence
\[
        [\mathcal J(\mathcal H)]
        \,=\,(u+v)^n
        \,=\,\sum_{a=1}^n \binom{n}{a}u^av^{n-a}.
\]
Thus \(\gamma_a(\mathcal J(\mathcal H))=\binom{n}{a}\), in agreement with
Theorem~\ref{thm:generic}.
\end{example}

We use the standard convention that the degree of a non-equidimensional
projective variety is the degree of its top-dimensional component.  The generic fiber dimension~of~\eqref{eq:pi}~is:
\[
        f=d-1-\dim\mathcal E(\mathcal H),
\]
where \(\dim\mathcal E(\mathcal H)\) denotes the largest dimension of an
irreducible component of $\mathcal E(\mathcal H)$.

\begin{theorem}
\label{thm:degreefromclass}
Under the assumptions above, we have
\[
        f=d-\max\{a:\gamma_a(\mathcal J(\mathcal H))>0\}
        \quad {\rm and} \quad
        \deg\mathcal E(\mathcal H)
        =\gamma_{d-f}(\mathcal J(\mathcal H)).
\]
\end{theorem}

\begin{proof}
For a component \(Y_j\), set \(X_j=\pi_2(Y_j)\), and let \(f_j\) be the
fiber dimension of the map \(Y_j\to X_j\) over a general point of \(X_j\).
Since \(\dim Y_j=d-1\), we have \(\dim X_j=d-1-f_j\).  The generic fibers of
\(\mathcal J(\mathcal H)\to\PP^{n-1}\) are projective linear spaces in the
first factor, because the eigenvector equations are linear in
\((t :\lambda)\). By \eqref{eq:gamma-intersection} the largest
power of \(u\) appearing in \([Y_j]\) is \(u^{d-f_j}\).  Indeed, if fewer
than \(f_j\) hyperplanes are imposed in the first factor, a general fiber is
not cut to finitely many points; if exactly \(f_j\) are imposed, a general
fiber is cut to one point, and the remaining \(d-1-f_j\) hyperplanes in the
second factor compute \(\deg X_j\).  Hence
\[
        \gamma_{d-f_j}(Y_j)\,=\,\deg X_j.
\]
The top-dimensional images are those with minimal \(f_j\), namely
\(f_j=f\).  The largest exponent of \(u\) in
\([\mathcal J(\mathcal H)]\) is \(d-f\), and the coefficient of
\(u^{d-f}v^{n-d+f}\) is the sum of the degrees of the top-dimensional images.
Since these images are assumed distinct, this sum is
\(\deg\mathcal E(\mathcal H)\).
\end{proof}

The same Chow class also gives the degrees obtained from general linear
subspaces of \(\mathcal H\). 
The irreducibility assumption below is important. When the horizontal incidence variety \(\mathcal J(\mathcal H)\) is reducible, lower-dimensional images may become visible after restriction.  

\begin{proposition}\label{prop:genericsubspaces}
Assume that \(\mathcal J(\mathcal H)\) is irreducible, and let
\(\mathcal L_c\subset\mathcal H\) be a general subspace of codimension \(c\).
If \(0\leq c\leq f\), then
\(\mathcal E(\mathcal L_c)=\mathcal E(\mathcal H)\).  If
\(c=f+i\), with \(0\leq i\leq d-1-f\),~then
\[
        \dim\mathcal E(\mathcal L_{f+i})=d-1-f-i,
        \qquad
        \deg\mathcal E(\mathcal L_{f+i})
        =
        \gamma_{d-f-i}(\mathcal J(\mathcal H)).
\]
\end{proposition}

\begin{proof}
Over a dense open subset of \(\mathcal E(\mathcal H)\), the fibers of \eqref{eq:pi} are $f$-dimensional projective subspaces
in the first factor.  If \(c\leq f\), then a general codimension
\(c\) linear subspace of the first factor meets these fibers, so
\(\mathcal E(\mathcal L_c)\) contains a dense open subset of \(\mathcal E(\mathcal H)\).  Since
\(\mathcal E(\mathcal L_c)\subseteq \mathcal E(\mathcal H)\), the first assertion follows.
Now take \(c=f+i\), and set
\(
        Y_c=
        \mathcal J(\mathcal H)\cap\bigl(\PP(\mathcal L_c \oplus \CC)\times\PP^{n-1}\bigr).
\)
For general \(\mathcal L_c\), this intersection has class \(u^c\cap[\mathcal J(\mathcal H)]\),
dimension \(d-1-c=d-1-f-i\), and maps generically one-to-one onto
\(\mathcal E(\mathcal L_c)\).  This implies
\[
        \deg\mathcal E(\mathcal L_c)
       \, =\,
        \int_{Y_c}v^{d-1-f-i}
       \, =\,
        \int_{\mathcal J(\mathcal H)}u^{f+i}v^{d-1-f-i}.
\]
This integral extracts the coefficient \(\gamma_{d-f-i}(\mathcal J(\mathcal H))\)
from the class \([\mathcal J(\mathcal H)]\).
\end{proof}

\begin{example}
In Examples~\ref{ex:skew}, \ref{ex:skew2} and \ref{ex:skew3} we considered the Lie algebra \(\mathfrak{so}_4\).  Here,
\(d=6\), \(n=4\) and \(\mathcal J(\mathfrak{so}_4)\) is irreducible.  We compute
$\,
 [\mathcal J(\mathfrak{so}_4)]
 =
 4uv^3+4u^2v^2+2u^3v
 \in A^4(\PP^6\times \PP^3)$.
The largest exponent of \(u\) is \(3\), so \(f=6-3=3\).  Hence
\(\mathcal J(\mathfrak{so}_4)\to\mathcal E(\mathfrak{so}_4)\) has 
fiber dimension \(3\), and $
 \deg\mathcal E(\mathfrak{so}_4)=\gamma_3=2$.
This is the isotropic quadric in \(\PP^3\).  Proposition
\ref{prop:genericsubspaces} says that a general codimension \(3\) subspace
of \(\mathfrak{so}_4\) has the same eigenvector surface.  General subspaces
of codimension \(4\) and \(5\) give, respectively, a curve of degree
\(\gamma_2=4\) and \(\gamma_1=4\) points.
\end{example}

\begin{example}
\label{ex:revisit42}
We revisit Example~\ref{ex:addsym}, the case \(k=4,m=2\) in
Theorem~\ref{thm:sympower}.  Here \(d=4\), \(n=5\), and
\(\mathcal J(\mathcal H)=Y_1\cup Y_2\cup Y_3\).  The images of \(Y_1\) and
\(Y_2\) are the two surface components of \(\mathcal E(\mathcal H)\), while
\(\pi_2(Y_3)\) is a curve contained in their union.  The
Chow class decomposes as
\[
\begin{aligned}
[\mathcal J(\mathcal H)]
&\,=\,
5uv^4+10u^2v^3+10u^3v^2  \\
&\,=\,
(uv^4+2u^2v^3+4u^3v^2)
+
(2uv^4+4u^2v^3+6u^3v^2)
+
(2uv^4+4u^2v^3).
\end{aligned}
\]
The largest exponent of \(u\) is \(3\), so \(f=4-3=1\).  The coefficient
\(\gamma_3=10=4+6\) is the degree of the surface
\(\mathcal E(\mathcal H)\), whose two components have degrees \(4\) and
\(6\).  After passing to a general codimension $2$ subspace, the
top-dimensional eigenvector variety is a curve of degree
\(\gamma_2=10=2+4+4\).  The last summand comes from the previously hidden
component \(Y_3\). Proposition \ref{prop:genericsubspaces} still holds, because the images of $Y_1,Y_2$ and $Y_3$ under $\pi_2$ are all distinct. 
A general codimension \(3\) subspace gives
\(\gamma_1=5\) points, as expected for a generic \(5\times5\) matrix.
\end{example}

We close by presenting a vector bundle method for computing the class
\eqref{eq:classesadditive}.  Let \(Y\) be an irreducible component of
\(\mathcal J(\mathcal H)\) and let \(X=\pi_2(Y)\subset\PP^{n-1}\). We denote the tautological line bundle by
\(L=\mathcal O_{\PP^{n-1}}(-1)|_X\), so
\(L_{[x]}=\CC x\).  Consider the bundle map
\begin{equation}
\label{eq:eigenbundlemap}
 \Phi_{\mathcal H}:
 (\mathcal H \oplus \CC)\otimes\mathcal O_X
 \longrightarrow
 \operatorname{Hom}(L,\CC^n\otimes\mathcal O_X)
 \simeq L^\vee\otimes\CC^n,
\end{equation}
whose fiber at \([x]\) sends \((t,\lambda)\) to the map
\(s\mapsto H(t)s-\lambda s\).  Its fiberwise kernel is
\[
        K_{[x]}=
        \{\,(t,\lambda)\in\mathcal H \oplus \CC:H(t)x=\lambda x\,\}.
\]
We say that \(Y\) is of \emph{eigenbundle type} if these kernels have
constant dimension on \(X\), and if \(Y\simeq\PP_X(K_{\mathcal H})\), where
\(K_{\mathcal H}=\ker(\Phi_{\mathcal H})\).  If the generic fiber dimension
of \(p:Y\to X\) is \(f\), then \(\operatorname{rank}K_{\mathcal H}=f+1\).
Let
$        h=c_1(\mathcal O_X(1))$ and
$      \xi=c_1(\mathcal O_{\PP_X(K_{\mathcal H})}(1)).
$
Under the embedding \(Y\subset\PP^d\times X\), the classes \(\xi\) and \(h\)
are the restrictions of \(u\) and \(v\).  We substitute $a=\dim X+1-j$ in \eqref{eq:gamma-intersection} and use $p_*(\xi^{f+j})=s_j(K_{\mathcal H})$, where \(s_j(K_{\mathcal H})\) is the \(j\)-th Segre class. This gives
\begin{equation}
\label{eq:segrecoefficients}
        \gamma_{\dim X+1-j}(Y)
        \,\,=\,\,
        \int_X s_j(K_{\mathcal H})h^{\dim X-j}
        \,\,=\,\,\deg_X(s_j(K_{\mathcal H})),
        \qquad 0\leq j\leq \dim X.
\end{equation}
We now return to the Lie algebra setting of Section \ref{sec5}.

\begin{proposition}\label{prop:minuscule-degree-formula}
Let $\mathfrak g$ be
a minuscule representation, and put
\(f=\dim\mathfrak g-1-\dim \mathcal{E}(\mathfrak g)\).  If
the subspace
\(\mathcal L_{f+i}\subset\mathfrak g\) is general of codimension \(f+i\)
 then
\[ \quad
        \deg\mathcal E(\mathcal L_{f+i})
        \,=\,
        \deg_{\mathcal{E}(\mathfrak g)} c_i(T_{\mathcal{E}(\mathfrak g)}) \qquad
        {\rm for} \,\, \,\,0\leq i\leq\dim {\mathcal{E}(\mathfrak g)}.
\]
\end{proposition}

\begin{proof}
By Proposition~\ref{thm:minuscule}, the eigenvector variety \(\mathcal{E}(\mathfrak g)\) is the closed \(G\)-orbit of a highest weight vector. Hence the kernels \(K_{[x]}\) have constant dimension and
\(\mathcal J(\mathfrak g)\simeq\PP_{\mathcal{E}(\mathfrak g)}(K_{\mathfrak g})\).  Since
\(d-f=\dim \mathcal{E}(\mathfrak g)+1\), the formula \eqref{eq:segrecoefficients} gives
$
        \gamma_{d-f-i}(\mathcal J(\mathfrak g))
        =
        \deg_{\mathcal{E}(\mathfrak g)} s_i(K_{\mathfrak g}).
$
It remains to identify these Segre classes.  Let \(M_{\mathfrak g}\) be the
image of the map \(\Phi_{\mathfrak g}\) in \eqref{eq:eigenbundlemap}.  Its fiber is
\(\CC x+\mathfrak g\cdot x\), the affine tangent space to the cone over
\(\mathcal{E}(\mathfrak g)\) at \(x\).  Therefore, we have the exact sequence
\[
        0\longrightarrow\mathcal O_{\mathcal{E}(\mathfrak g)}
        \longrightarrow M_{\mathfrak g}
        \longrightarrow T_{\mathcal{E}(\mathfrak g)}
        \longrightarrow 0.
\]
On the other hand, the following sequence has trivial middle term:
\[
        0\longrightarrow K_{\mathfrak g}
        \longrightarrow (\CC\oplus\mathfrak g)\otimes\mathcal O_{\mathcal{E}(\mathfrak g)}
        \longrightarrow M_{\mathfrak g}
        \longrightarrow 0.
\]
Hence
\(s(K_{\mathfrak g})=c(M_{\mathfrak g})=c(T_{\mathcal{E}(\mathfrak g)})\).  The result now follows
from Proposition~\ref{prop:genericsubspaces}.
\end{proof}

\begin{example}[Grassmannians]
Let \(\mathcal A_{k,m}\) be the space of additive compound matrices from
Corollary~\ref{cor:addcomp}.  Its eigenvector
variety is the Grassmannian in its Pl\"ucker embedding.
Let \(\mathcal S\) and \(\mathcal Q\) be the
tautological subbundle and quotient bundle on \(\mathrm{Gr}(k,m)\). Then $\,
T_{\mathrm{Gr}(k,m)}\simeq \mathcal S^\vee\otimes\mathcal Q$.
We abbreviate $\delta=k(m-k) = {\rm dim}(\mathrm{Gr}(k,m))$.
Proposition~\ref{prop:minuscule-degree-formula} gives
\begin{equation}
\label{eq:grassmannianchernformula}
        \gamma_{\delta+1-j}(\mathcal J(\mathcal A_{k,m}))
        \,\,=\,\,
        \int_{\mathrm{Gr}(k,m)}
        c_j(T_{\mathrm{Gr}(k,m)})h^{\delta-j},
        \qquad 0\leq j\leq\delta,
\end{equation}
where \(h=c_1(\mathcal O_{\mathrm{Gr}(k,m)}(1))\).  In particular,
using $c_1(T_{\mathrm{Gr}(k,m)})=mh$, we conclude
\[
        \gamma_{\delta+1}
        \,\,=\,\,
        \deg\mathrm{Gr}(k,m)
        \,\,=\,\,
        \frac{\delta!\prod_{a=1}^{k-1}a!}
             {\prod_{a=0}^{k-1}(m-k+a)!},
        \qquad
        \gamma_{\delta}\,=\,m\,\deg\mathrm{Gr}(k,m).
\]
For $k=2$, $m=4$, the Grassmannian is a smooth quadric in
\(\PP^5\).  The normal sequence gives
\[
        c(T_{\mathrm{Gr}(2,4)})
        \,=\,
        \frac{(1+h)^6}{1+2h}
       \, =\,
        1+4h+7h^2+6h^3+3h^4,
        \qquad
        \int_{\mathrm{Gr}(2,4)}h^4=2.
\]
Hence the Chow coefficients are \(2,8,14,12,6\). Later in \eqref{eq:onebodyH} we consider additive compound matrices $T$ arising from symmetric $m \times m$ matrices. They are the fermionic one-body operators. This is a 
\(10\)-dimensional subspace of $\mathcal{A}_{k,m}$. By the same vector bundle computation,
\[
 [\mathcal J(T)]
 \,=\,
 6uv^5+12u^2v^4+14u^3v^3+8u^4v^2+2u^5v
\, \in\, A^6(\PP^{10}\times\PP^5).
\]
The largest exponent of \(u\) is \(5\), so \(f=10-5=5\).  Thus a general
codimension \(5+i\) subspace produces linear sections of \(\mathrm{Gr}(2,4)\), with degrees
$2, 8,14, 12, 6$.
For the full space \(\mathcal A_{2,4}\), the displayed
polynomial lies  in \(A^6(\PP^{16}\times\PP^5)\), and the fiber
dimension is
\(f=16-5=11\).
\end{example}

 \section{Quantum Physics}
\label{sec7}

Hamiltonians in quantum physics have highly special structure. They belong to  linear families
of very large square matrices with few
 degrees of freedom. This special structure puts restrictions on the eigenvectors of these matrices.
This was our motivation for this~paper.

We begin with fermionic Hamiltonians, which model electronic systems,
such as those described in \cite[Section 4]{FSS}. In that setting,
let $k$ denote the number of electrons and $m$ the number of orbitals. By particle-hole symmetry we may assume  that $2k \le m$, see \cite[Proposition~3.7]{FSS}.
We express the
 fermionic Hamiltonian in second quantized form \cite{Sve}:
\begin{equation}\label{eq:fH}
    H \,\,\,=\,\, T \,+ \,W \,\, = \,\,\sum_{p,q = 1}^m h_{p,q} \,a_p^\dagger a_q \,\,\,+ \!\!
    \sum_{1 \le p < r, q < s \le m} \!\!\!\!\! w_{pr, qs} \, a_p^\dagger a_r^\dagger a_s a_q.
\end{equation}
The $m \times m$ matrix  $h=(h_{p,q})$  and 
the $\binom{m}{2} \times \binom{m}{2}$ matrix $w=(w_{pr,qs})$ are symmetric and have unknown entries.
Here $a_p$ and $a_p^\dagger$ are the fermionic creation and annihilation operators. 
They generate the Fermi--Dirac algebra.  An introduction can be found in~\cite[Section~2]{Sve}.

The Hamiltonian $H$ acts on the exterior algebra $\wedge \CC^m$.
This action restricts to the exterior power $\wedge^k \CC^m$ for any fixed $k$,
because each summand 
in (\ref{eq:fH}) contains equally many creation and annihilation operators.
We therefore view $H = T+W$ as an $\binom{m}{k} \times \binom{m}{k}$ matrix. 

\begin{example}[$k=2,m=4$]
Here $H$ is a $6 \times 6$ matrix with rows and columns indexed by 
the $2$-element subsets $\{12,13,14,23,24,34\}$.
The fermionic one-body operator equals
\begin{equation}\label{eq:onebodyH}
T \,\,\,= \,\,\,\begin{bmatrix}
h_{11}\!+\!h_{22} & h_{23} & h_{24} & -h_{13} & -h_{14} & 0 \\
h_{23} & h_{11}\!+\!h_{33} & h_{34} & h_{12} & 0 & -h_{14} \\
h_{24} & h_{34} & h_{11}\!+\!h_{44} & 0 & h_{12} & h_{13} \\
-h_{13} & h_{12} & 0 & h_{22}\!+\!h_{33} & h_{34} & -h_{24} \\
-h_{14} & 0 & h_{12} & h_{34} & h_{22}\!+\!h_{44} & h_{23} \\
0 & -h_{14} & h_{13} & -h_{24} & h_{23} & h_{33}\!+\!h_{44}
\end{bmatrix}.
\end{equation}
The characteristic polynomial $\chi_H$ is irreducible 
in $\CC[\lambda,h]$ of degree $6$, and it has $956$ terms. 
In this small case,  the fermionic two-body operator is
a general  symmetric $6 \times 6$ matrix:
\begin{equation}\label{eq:twobodyW}
W \,\,\,= \,\,\,\begin{bmatrix}
\, w_{12, 12} & \, w_{12, 13} & \, w_{12, 14} & \, w_{12, 23} & \, w_{12, 24} & \, w_{12, 34} \,\, \\
\, w_{12, 13} & \, w_{13, 13} & \, w_{13, 14} & \, w_{13, 23} & \, w_{13, 24} & \, w_{13, 34} \, \, \\
\, w_{12, 14} & \, w_{13, 14} & \, w_{14, 14} & \, w_{14, 23} & \, w_{14, 24} & \, w_{14, 34} \, \, \\
\, w_{12, 23} & \, w_{13, 23} & \, w_{14, 23} & \, w_{23, 23} & \, w_{23, 24} & \, w_{23, 34} \,\, \\
\, w_{12, 24} & \, w_{13, 24} & \, w_{14, 24} & \, w_{23, 24} & \, w_{24, 24} & \, w_{24, 34} \,\, \\
\, w_{12, 34} & \, w_{13, 34} & \, w_{14, 34} & \, w_{23, 34} & \, w_{24, 34} & \, w_{34, 34} \,\, \\
\end{bmatrix}.
\end{equation}
\end{example}

The special structure of the fermionic Hamiltonian $H=T+W$ appears only for larger values of $k$ and $m$. In Theorem \ref{thm:ferm1body} we identify the eigenvector variety 
$\mathcal{E}(T)$.


\begin{proof}[Proof of Theorem \ref{thm:ferm1body}]
    The fermionic one-body operators $T$ are additive compound matrices:
\begin{equation}
\label{eq:exterioraction2}
    T \cdot (v_1 \wedge \cdots \wedge v_k) \,\,= \,\,
    \sum_{p,q=1}^m h_{p,q} \, a_p^\dagger a_q v_1 \wedge \cdots \wedge v_k 
    \,\,=\,\, \sum_{\ell = 1}^k v_1 \wedge \cdots \wedge h v_\ell \wedge \cdots \wedge v_k.
\end{equation}
    See (\ref{eq:exterioraction}).
     Hence the matrices $T$ form a subspace of $\mathcal{A}_{k,m}$.  In Lie theory notation, this subspace is
     the image of the map $\mathrm{Sym}^2(\CC^m) \to \mathfrak{gl}(\wedge^k \CC^m)$.
          Corollary \ref{cor:addcomp} implies that the eigenvector variety
     $\mathcal{E}(T)$ is contained in $ \mathcal{E}(\mathcal{A}_{k,m}) = 
     \mathrm{Gr}(k,m)$.      We claim that  equality holds.
           
            The space of symmetric matrices $\mathrm{Sym}^2(\CC^m) \subseteq \mathfrak{gl}_m$ is not a Lie algebra. 
            However, the orthogonal group $\mathrm{O}_m$ acts on $\mathrm{Sym}^2(\CC^m)$ by conjugation, and a generic symmetric matrix is conjugate to a diagonal matrix. 
            The proof of Theorem~\ref{thm:intro-lie} applies in this setting. 
            Since $\mathrm{O}_m$ acts transitively on nondegenerate $k$-dimensional subspaces of $\CC^m$,
            we conclude that 
    the eigenvector variety $\mathcal{E}(T)$ is
     $\overline{\mathrm{O}_m \cdot [e_1 \wedge \cdots \wedge e_k]} = \mathrm{Gr}(k,m)$.
\end{proof}

Next we consider the two-body operator \(W\). Every
one-body operator is also represented by a two-body operator. The embedding sends
\(h\in\CC^{m \times m}\) to the following operator on \(\wedge^2\CC^m\)
\[
   \iota(h)(u\wedge v)\,=\,\frac{1}{k-1}(hu\wedge v+u\wedge hv).
\]
Substituting \(\iota(h)\) into the two-body construction below gives the earlier action of \(h\) on \(\wedge^k\CC^m\). Thus  \(H=T+W\) is itself a
two-body operator on \(\wedge^k\CC^m\). We can write \(W\) explicitly as 
\begin{equation}\label{eq:2bodyop}
  \! W\,:\, \wedge^k \CC^m \to \wedge^k \CC^m, 
   \,\,\, v_1 \wedge \cdots \wedge v_k \,\,\mapsto \sum_{1 \le j < \ell \le k}\! (-1)^{\ell - j - 1} v_1 \wedge \cdots \wedge w(v_j \wedge v_\ell) \wedge \cdots \wedge v_k.
\end{equation}
 This construction is analogous to the additive compound matrices
 in (\ref{eq:exterioraction}) and (\ref{eq:exterioraction2}).
   Note that the case $k=1$ is void, 
and  $W = w$ if $k = 2$, as in (\ref{eq:twobodyW}).
From now on, we assume $k \geq 3$.

Let $\mathcal{W}_{k,m}$ denote the space of two-body operators.
This is isomorphic to ${\rm Sym}^2( \wedge^2 \CC^m)$, so its
dimension equals $\binom{\binom{m}{2} + 1}{2}$. It is a subspace
of ${\rm Sym}^2(\wedge^k \CC^m)$.
Like the space of one-body operators, 
$\mathcal{W}_{k,m}$ is not a Lie algebra. But,
its eigenvector variety is even more intricate.



\begin{proposition} \label{prop:2body}
  The eigenvector variety $\mathcal{E}(\mathcal{W}_{k,m})$  is irreducible, and its dimension satisfies
  \begin{equation}
  \label{eq:twobodydim}
  {\rm dim} \,\mathcal{E}(\mathcal{W}_{k,m}) \,\,\leq \,\,
    \min\left(\binom{m}{k} - 1\,, \binom{\binom{m}{2}+1}{2}-2\right).
  \end{equation}
  In particular,   if $k \ge 7$ and $m \ge 15$ then
  $\mathcal{E}(\mathcal{W}_{k,m})$ is a proper subvariety of  $\,\PP\bigl(\wedge^k \CC^m\bigr)$.
\end{proposition}

\begin{proof}
The eigenvector variety lives in $\PP(\wedge^k\CC^m)$, so its dimension is at most
$\binom{m}{k}-1$. The second upper bound comes from Corollary~\ref{cor:d-2}: the
space $\mathcal W_{k,m}$ contains the identity matrix and has dimension
$\binom{\binom{m}{2}+1}{2}$. The displayed inequality is strict for $k\ge7$ and $m\ge15$.

It remains to justify irreducibility. By the inclusion of one-body operators into the
two-body space, the restriction of $\chi_{\mathcal W_{k,m}}$ to that linear subspace is
the characteristic polynomial $\chi_T$. The one-body operators define a minuscule representation. Hence the Weyl group of $\mathrm{O}_m$ acts transitively on the weights. Since $\chi_T$ and $\chi_{\mathcal W_{k,m}}$ are both nonzero homogeneous polynomials, then every decomposition of the latter results in a decomposition of the former. 
This implies $\chi_{\mathcal W_{k,m}}$ is also irreducible.
Proposition~\ref{prop:dominant-components}
then gives an irreducible horizontal incidence variety, and its image
$\mathcal E(\mathcal W_{k,m})$ is irreducible.
\end{proof}

\begin{conjecture}\label{conj:dim}
Equality holds in~\eqref{eq:twobodydim}. Equivalently, the generic fiber of the projection 
$\mathcal{J}(\mathcal{W}_{k,m}) \to \mathcal{E}(\mathcal{W}_{k,m})$ is one-dimensional, spanned 
by the identity direction in $\mathcal{W}_{k,m}$.
\end{conjecture}

This conjecture is supported by the following numerical evidence.
We compute the dimension of $\mathcal{E}(\mathcal{W}_{k,m})$ by taking a generic matrix 
in $\mathcal{W}_{k,m}$ and computing an eigenvector $x$ via a standard eigenvalue 
solver. We then determine the rank of the matrix $M(x)$ evaluated at $x$ using Julia~\cite{julia}. For $k$ and $m$ where the dimension of $\mathcal W_{k,m}$ does not exceed that 
of $\mathbb{P}(\wedge^k \mathbb{C}^m)$, numerical computations show that $\mathcal E(\mathcal W_{k,m})$ fills 
the whole ambient space.
The results where we obtain proper subvarieties are displayed in 
Table~\ref{tab:conjecture}. Our Julia code is available at \cite{code}.

\begin{table}[H]
\centering
\renewcommand{\arraystretch}{1.2}
\begin{tabular}{ccccc}
\hline
$k$ & $m$ & $\dim \mathbb{P}(\wedge^k \mathbb{C}^m)$ & 
$\dim \mathcal W_{k,m}$ & $\dim \mathcal E(\mathcal W_{k,m})$ \\
\hline
7 & 15 & 6434 & 5565 & 5563 \\
6 & 16 & 8007 & 7260 & 7258 \\
7 & 16 & 11439 & 7260 & 7258  \\
8 & 16 & 12869 & 7260 & 7258  \\
6 & 17 & 12375 & 9316 & 9314 \\
7 & 17 & 19447 & 9316 & 9314 \\
8 & 17 & 24310 & 9316 & 9314 \\
\hline
\end{tabular}
\caption{Numerical verification of Conjecture~\ref{conj:dim}. 
The computed dimension in the last column is $\dim \mathcal W_{k,m} - 2$, 
the prediction of~\eqref{eq:twobodydim}, with generic fiber dimension exactly one.}
\label{tab:conjecture}
\end{table}

We next consider bosonic systems, where $k$ is the number of bosons and $m$ is the number of sites. 
The general one-body operator $S$ arises from a symmetric $m \times m$ matrix $h = (h_{i,j})$:
\begin{equation} 
\label{eq:symH} S \,=\, \sum_{i,j = 1}^m h_{i,j} \,b_i^\dagger b_j \, \qquad {\rm where} \quad
b_i^\dagger \mapsto z_i \,\,\, {\rm and}\,\,\, b_i \mapsto \frac{\partial}{\partial z_i}. 
\end{equation}
The bosonic operators generate a Weyl algebra, which 
acts on $\mathrm{Sym}^k(\CC^m)$ as in (\ref{eq:operator}). Again,
this is an $\binom{m+k-1}{k} \times \binom{m + k - 1}{k}$ matrix, but now
its entries are linear forms in $\binom{m+1}{2}$ variables~$h_{i,j}$.

\begin{theorem} \label{thm:sympower2}
The eigenvector variety for bosonic one-body operators is the union of the
projective $\mathrm{SO}_m$-orbit closures of degree $k$ monomials. This is an irreducible decomposition
\begin{equation}
\label{eq:sympower2}
\cE(S) \,=\, \bigcup_{\substack{\mu \vdash k\\ \ell(\mu) \leq m}}
\overline{\mathrm{SO}_m\cdot [z^\mu]}.
\end{equation}
\end{theorem}
\begin{proof}
The proof is analogous to the proof of Theorem~\ref{thm:sympower}. Generic symmetric matrices are
orthogonally diagonalizable on a dense open set, so the proof of Theorem~\ref{thm:intro-lie} applies. The monomials are the weight vectors
for the diagonal torus, and hence the eigenvector variety consists of their $\mathrm{O}_m$-orbit closures. These orbits coincide with the $\mathrm{SO}_m$-orbits, since the extra component differs from \(\mathrm{SO}_m\) by a reflection that fixes the monomial projectively. This gives \eqref{eq:sympower2}. 
All the orbit closures are irreducible. They are also not contained in one another, since
distinct orthogonal factors cannot specialize to the same nonisotropic factor. Hence \eqref{eq:sympower2} is an~irreducible~decomposition.
\end{proof}

In quantum physics, one is most interested in the 
{\em ground state} of the Hamiltonian.  This is the
eigenvector corresponding to the smallest eigenvalue.
If  $a_1 \le \dots \le a_m$ are the eigenvalues of the symmetric matrix $h$, then
$k a_1$ is the smallest eigenvalue of \eqref{eq:symH}.
The corresponding eigenvector lies in the 
${\rm SO}_m$-orbit of the pure power $z_1^k$, which 
 is the Veronese variety $\nu_k(\PP^{m - 1}) = \overline{{\rm SO}_m\, z_1^k}$.
This corresponds to the partition $\mu = (k,0,\ldots,0)$ in (\ref{eq:sympower2}).

\begin{corollary}\label{cor:veronese}
The Veronese variety $\nu_k(\PP^{m - 1}) $ is the  irreducible component of the eigenvector variety \eqref{eq:sympower2}
that corresponds to the ground states of real bosonic one-body~operators~(\ref{eq:symH}).
\end{corollary}

We conclude with a well-known physics model, namely
 the {\em Bose-Hubbard Hamiltonian}:
\begin{equation}\label{eq:BH}
    H \,\,=\,\, T + W \,\, = \,\, -t\sum_{i = 1}^{m - 1} (b_i^\dagger b_{i + 1} + b_{i + 1}^\dagger b_{i}) \,+\, \sum_{i = 1}^m \varepsilon_i b_i^\dagger b_i \,+\, \frac{U}{2} \sum_{i = 1}^m b_i^\dagger b_i^\dagger b_i b_i.
\end{equation}
See e.g.~\cite{FM}.
As in (\ref{eq:fH}),
this is the sum of a one-body operator (for $U=0$) and a two-body operator $W = H-T$.
Note that $T$ is the specialization of (\ref{eq:symH}) to the 
$m \times m$ {\em Jacobi matrix}
\begin{equation}\label{eq:onepartBH} h \,\, = \,\,
    \begin{bmatrix}
    \varepsilon_1 & -t & 0 & \cdots & 0\\
    -t & \varepsilon_2 & -t & \cdots & 0\\
    0 & -t & \varepsilon_3 & \cdots & 0\\
    \vdots & \vdots & \vdots & \ddots & \vdots\\
    0 & 0 & 0 & \cdots & \varepsilon_m\\
\end{bmatrix}.
\end{equation}
If $m = 2$ then $h$ is a general symmetric matrix. By
Theorem \ref{thm:sympower2}, the eigenvector variety
is a curve in $\PP^k$ with $\lfloor \frac{k+2}{2}\rfloor$
irreducible components, indexed by 
$\mu =(i,k-i)$ with $2i \geq k$.

\begin{example}[$m\!=\!2,k\!=\!4$]
This is a variant of Example \ref{ex:addsym}.
The eigenvector curve of $T$ has three irreducible components in $\PP^4$.
They are  given parametrically by the binary quartics 
$$(\alpha z_1 + \beta z_2)^4 \, ,\quad
 (\alpha z_1 + \beta z_2)^3 (\beta z_1 - \alpha z_2) \quad {\rm and} \quad
 (\alpha z_1 + \beta z_2)^2 (\beta z_1 - \alpha z_2)^2 . $$
For instance, the prime ideal of the third component is
$\langle x_1-x_5, x_2+x_4, x_4^2-4x_3 x_5-8 x_5^2 \rangle$.
All three curves lie on the eigenvector surface $\mathcal{E}(H)$
of the full Bose-Hubbard Hamiltonian $H$. That irreducible surface
has degree $10$ in $\PP^4$, and its ideal is generated by five quartics.
\end{example}

If $m \geq 3$ then $h$ is a special symmetric matrix, and
$\mathcal{E}(T)$ is strictly contained  in $\cE(S)$,~\eqref{eq:sympower2}.

\begin{example}[$m=k=3$]
This is a variant of Example \ref{ex:33eigenvectors}.
The eigenvector variety $\cE(S)$ from (\ref{eq:sympower2}) lives in $\PP^9$.
It has three irreducible components in $\PP^9$, namely the Veronese surface
$\overline{\mathrm{SO}_3\cdot [z_1^3]}$ and the two
threefolds  $\overline{\mathrm{SO}_3\cdot [z_1^2z_2]}$ and $\overline{\mathrm{SO}_3\cdot [z_1z_2z_3]}$.
The last one is the
classical {\em variety of self-polar triangles} \cite[Section 2.1.3]{Dol}.
For the Bose-Hubbard one-body operator $T$, $\mathcal{E}(T)$ has  three components of dimension two.
One of them is  the Veronese surface, and the other two are
surfaces of degrees $26$ and $5$ inside the threefolds above.
 The variety $\mathcal{E}(H)$ for the full Bose-Hubbard Hamiltonian $H$ is an irreducible
threefold of degree $178$ in $\PP^9$.
\end{example}

We saw that eigenvector varieties for Hamiltonians of 
quantum systems are both intrinsically interesting and useful for 
applications. Our work opens several directions for further study. 
On the theoretical side, the most pressing one is 
Conjecture~\ref{conj:dim}, which would give a sharp dimension formula 
for the eigenvector variety of the fermionic two-body operator. More 
broadly, it would be interesting to learn which projective varieties 
arise as eigenvector varieties of some linear space $\cH$, and to 
understand how the geometry of $\mathcal{E}(\cH)$ reflects the 
algebraic structure of $\cH$. On the applied side, the connection to 
coupled cluster theory~\cite{FSS} suggests that eigenvector varieties 
may encode obstructions to efficient approximation of ground states, a topic we plan to pursue. Finally, the bosonic two-body 
operator, the analogue of $\mathcal{W}_{k,m}$ in the bosonic 
setting, was not studied here and presents a natural next~step.

\bigskip \bigskip

\noindent
{\bf Acknowledgements.} The third author thanks
KTH Stockholm for hosting her during this project.
We are grateful to Otto Schmidt for helping us with the
bosonic Hamiltonians.

\bigskip
\bigskip
		
\footnotesize
\noindent {\bf Authors' addresses:}
		
\noindent Sandra Di Rocco, KTH Stockholm \hfill \url{dirocco@math.kth.se}
				
\noindent  Bernd Sturmfels, MPI-MiS Leipzig   \hfill \url{bernd@mis.mpg.de}

\noindent Svala Sverrisd\'ottir, MPI-MiS Leipzig \hfill \url{sverris@mis.mpg.de}

\end{document}